\documentclass{article}

\usepackage[preprint,nonatbib]{Preprint_MMMS}
\usepackage{amssymb}
\usepackage{amsthm}
\usepackage{mathtools}
\usepackage{amsfonts}
\usepackage{amsmath}

\usepackage[utf8]{inputenc} 
\usepackage[T1]{fontenc}    
\usepackage{hyperref}       
\usepackage{url}            
\usepackage{booktabs}       
\usepackage{amsfonts}       
\usepackage{nicefrac}       
\usepackage{microtype}      
\usepackage{xcolor}         

\newtheorem*{theorem*}{Theorem}
\newtheorem{theorem}{Theorem}[section]

\newtheorem{proposition}[theorem]{Proposition}
\newtheorem{lemma}[theorem]{Lemma}

\newtheorem{corollary}[theorem]{Corollary}

\title{Concentration of the missing mass in metric spaces}

\author{%
	Andreas Maurer \\
	Istituto Italiano di Tecnologia, 16163 Genoa, Italy\\
	\texttt{am@andreas-maurer.eu}
}

\begin{document}

	\maketitle

	\begin{abstract}
We study the estimation and concentration on its expectation of the probability to observe data further than a specified distance from a given iid sample in a metric space. The problem extends the classical problem of estimation of the missing mass in discrete spaces. We give some estimators for the conditional missing mass and show that estimation of the expected missing mass is difficult in general. Conditions on the distribution, under which the Good-Turing estimator and the conditional missing mass concentrate on their expectations are identified. Applications to anomaly detection, coding, the Wasserstein distance between true and empirical measure and simple learning bounds are sketched.\end{abstract}
	
\section{Introduction}

What is the probability that lightning will strike more than a given
distance from one of the previously observed strikes? In the genetic survey
of some species, how large is the population of individuals, whose DNA\
differs from the previously observed sequences in more than a fixed number
of positions? Have we seen all handwritten digits up to some given
precision? Under the assumption of independent observations, these questions
and a number of similar problems can be formalized as follows.

In a metric probability space $\left( \mathcal{X},d,\mu \right) $ an iid
sample $\mathbf{X}=\left( X_{1},...,X_{n}\right) $ is drawn from $\mu $. For 
$r\geq 0$ we would like to estimate the \textit{conditional missing mass},
defined as the random variable%
\begin{equation*}
\hat{M}\left( \mathbf{X},r\right) =\mu \left\{ y:\forall i\in \left\{
1,...,n\right\} ,d\left( y,X_{i}\right) >r\right\} ,
\end{equation*}%
The conditional missing mass is the probability of finding a point at
distance more than $r$ from the given sample. The \textit{expected missing
	mass} is its expectation $M\left( \mu ,n,r\right) =\mathbb{E}\left[ \hat{M}%
\left( \mathbf{X},r\right) \right] $. It is a scale-dependent property of
the distribution $\mu $, and the conditional missing mass is a
scale-dependent property both of the distribution and the sample.

In the \textit{discrete case} $\mathcal{X}$ is at most countable, $d\left(
x,y\right) =1$ for $x\neq y$, and $r<1$. The pedagogical narrative
underlying the discrete case is that we have seen zebras six times, elefants
three times and a lions only once in independent sightings. What is the
probability of running into a yet unseen species on the next sighting? The
problem surfaced in a more serious context, when Alan Turing's team was
decyphering the enigma code during World War Two. They found what is now
called the Good-Turing estimator $G$, the relative number of species (or
words or letters) having been encountered only once. Soon Turing's co-worker
Good showed that $G$ has small bias, and various strong concentration
results for both $\hat{M}$ and $G$ have been established since (\cite%
{good1953population}, \cite{mcallester2000convergence}, \cite%
{mcallester2003concentration}, \cite{berend2013concentration} and \cite%
{ben2017concentration}, the latter being a particularly complete treatise).

In this paper we study the missing mass in the extended setting of metric
spaces or spaces with more general distortion functions, thus opening the
way to other applications. We show that in separable metric spaces the
conditional missing mass converges to zero almost surely (Proposition \ref%
{Proposition Convergence}), but the emphasis is on finite sample bounds.
Potential examples are coding, anomaly detection, estimating the support of
a distribution, or applications to ecology, when there is nearly a continuum
of species such as frequently mutating bacteria or viruses. Some application
are sketched in Section \ref{Section applications}.

It is clear, that the discrete case applies to neither of the initially
posed problems (lightnings, genes and handwritten digits). In the discrete
case the relation $d\left( x,y\right) \leq r<1$ implies $x=y$ and is
therefore transitive and an equivalence relation, partitioning the space
into species, words or numbers. In the general setting the relation $d\left(
x,y\right) \leq r$ is only reflexive and symmetric but not transitive. For
this reason only weaker results can be expected and sometimes obtained only
under additional conditions. In the discrete case the negative association
of occupancy counts can be exploited, but in the general case it is not even
clear, what should be defined as occupancy counts, and different techniques
are called for.

There does not seem to be not much literature on the missing mass in metric
spaces. One reference is \cite{berend2012missing}, where Section 4 gives a
bound on the rate of decrease of $M\left( \mu ,n,r\right) $ for totally
bounded metric spaces. In \cite{kontorovich2017nearest} this is combined
with a not quite correct application of the discrete-case results on the
concentration of $\hat{M}\left( \mathbf{X},r\right) $ to the general
environment of metric spaces  (\cite{kontorovich2017nearest},(16), probably
a typo). In \cite{hanneke2020universal} this is corrected and bounds on $%
\hat{M}\left( \mathbf{X},r\right) $ are obtained by reduction to the
discrete case of occupation numbers on a partition into sets of diameter
less than $r/2$. These results are asymptotic and targeted to show the
consistency of certain nearest-neighbor sample-compression algorithms.

A brief summary of our findings is the following: just as in the discrete
case the conditional missing mass converges to zero almost surely and an
extension of the Good-Turing estimator can be used to estimate $\hat{M}$,
but no uniformly valid exponential bounds are available at this point, in
fact such may not exist at all. Another simple estimator bounds $\hat{M}$
above with high probability, but with potentially large upward bias. This
estimator may be very useful whenever $\hat{M}$ is expected to be very
small. The estimation problem for the expected missing mass $M$, is more
difficult, and there is no uniform and universally valid bound for any
estimator. Exponential bounds and tight bounds on the variance of $\hat{M}$
exist, but depend on an intrinsic dimensionality of the distribution.

We conclude this section with a summary of notation. The next section
introduces our results in detail. Then follows a sketch of applications, and
a section containing the proofs.

\subsection{Notation and conventions\label{Section Notation and conventions}}

For $m\in \mathbb{N}$ we use the abbreviation $\left[ m\right] =\left\{
1,...,m\right\} $. The indicator of a set $A$ is denoted $\mathbf{1}A$, its
complement by $A^{c}$ and the difference $A\cap B^{c}$ with $A\backslash B$.
Both cardinality of sets and absolute value of reals are denote by bars $%
\left\vert \cdot \right\vert $. Vectors are written in bold letters. If $%
\mathbf{x}=\left( x_{1},...,x_{n}\right) \in \mathcal{X}^{n}$, $k\in \left[ n%
\right] $ and $y\in \mathcal{X}$ then the substitution $S_{y}^{k}\left( 
\mathbf{x}\right) \in \mathcal{X}^{n}$ is defined by%
\begin{equation*}
S_{y}^{k}\left( \mathbf{x}\right) =\left(
x_{1},...,x_{k-1},y,x_{k+1},...,x_{n}\right) ,
\end{equation*}%
and the deletion by $\mathbf{x}^{\backslash k}=\left(
x_{1},...,x_{k-1},x_{k+1},...,x_{n}\right) \in \mathcal{X}^{n-1}$.

Random variables are written in upper case letters, $\mathbb{E}$, and $%
\mathbb{V}$ are used for expectation and variance respectively and , $%
\mathbb{P}$ for the probability of events. $\left\Vert Y\right\Vert
_{p}:=\left( \mathbb{E}\left[ \left\vert Y\right\vert ^{p}\right] \right)
^{1/p}$ for real valued $Y$ and $p\geq 1$. If $Y$ is a random variable with
values in $\left[ 0,1\right] $ then we write the complementary variable $%
Y^{\perp }=1-Y$. The unit mass at a point $x$ will be denoted with $\delta
_{x}$.

On $\mathbb{R}^{D}$ the letter $\lambda $ is used for the Lebesgue measure
and $e_{1},e_{2},...,e_{D}$ for the canonical basis vectors.

Throughout $\left( \mathcal{X},d,\mu \right) $ is a Hausdorff space with
Borel-probability measure $\mu $ and a continuous distortion function $d:%
\mathcal{X\times X\rightarrow }\left[ 0,\infty \right) $ satisfying $d\left(
x,x\right) =0$ and $d\left( x,y\right) =d\left( y,x\right) $. If $d$ is
indeed a metric it will be specially mentioned.

For $r>0$ and $x\in \mathcal{X}$ we write $B\left( x,r\right) =\left\{
y:d\left( x,y\right) \leq r\right\} $. Note that $x\in B\left( y,r\right)
\iff y\in B\left( x,r\right) $. Often we write simply $B\left( x\right) $ if 
$r$ is understood and there is no ambiguity. A subset $S\subseteq \mathcal{X}
$ is called $r$-separated (for $r>0$) is $d\left( x,y\right) >r$ for all $%
x,y\in S$ with $x\neq y$. If $A\subseteq \mathcal{X}$ an $r$-net of $A$ is a
maximal $r$-separated subset of $A$.

$X_{1},...,X_{n},...$ is a sequence of independent random variables
distributed in $\mathcal{X}$ as $\mu $. For $m_{1},m_{2}\in \mathbb{N}$, $%
m_{1}<m_{2}$ we write $\mathbf{X}_{m_{1}}^{m_{2}}=\left(
X_{m_{1}},X_{m_{1}+1},...,X_{m_{2}}\right) \sim \mu ^{m_{2}-m_{1}+1}.$ With $%
\mathbf{X}$ we mean $\mathbf{X}=\mathbf{X}_{1}^{n}=\left(
X_{1},...,X_{n}\right) \sim \mu ^{n}$, when $n$ is understood.$\bigskip $

For $r\geq 0$ the \textit{conditional missing mass} is the $\left[ 0,1\right]
$-valued random variable%
\begin{equation*}
\hat{M}\left( \mathbf{X}_{1}^{n},r\right) =\mu \left( \bigcap_{k\in \left[ n%
	\right] }B\left( X_{k},r\right) ^{c}\right)
\end{equation*}%
and the \textit{expected missing mass} $M\left( \mu ,n,r\right) =\mathbb{E}%
\left[ \hat{M}\left( \mathbf{X}_{1}^{n},r\right) \right] $. It is often more
convenient to work with their "positive" counterparts, the \textit{%
	conditional envelope mass} 
\begin{equation*}
\hat{M}^{\perp }\left( \mathbf{X}_{1}^{n},r\right) =1-\hat{M}\left( \mathbf{X%
}_{1}^{n},r\right) =\mu \left( \bigcup_{k\in \left[ n\right] }B\left(
X_{k},r\right) \right)
\end{equation*}%
and the \textit{expected envelope mass} $M^{\perp }\left( \mu ,n,r\right)
=1-M\left( \mu ,n,r\right) $. When there is no ambiguity we omit the
dependences on $\mathbf{X},r,\mu $ and $n$.$\bigskip $

\section{Results}

In this section we first state results on the estimation of the conditional
missing mass by the extended Good-Turing estimator $G$ and give exponential
upper bounds on the conditional missing mass with a simple martingale-type
estimator.

Then we show that the estimation of the \textit{expected} missing mass is
more difficult and that there is no universal uniformly converging
estimator. We then give tight bounds on the variance of $\hat{M}\left( 
\mathbf{X},r\right) $ and $G\left( \mathbf{X},r\right) $ and exponential
concentration inequalities depending on an auxiliary statistic $h\left( 
\mathbf{X},r\right) $, which can be interpreted as an empirical local
packing number.

\subsection{The Good-Turing estimator and the conditional missing mass\label%
	{Section Good Turing estimator}}

By independence we have for any $k\in \left[ n\right] $%
\begin{equation*}
\mu \left\{ \bigcap_{i\in \left[ n\right] :i\neq k}B\left( X_{i}\right)
^{c}\right\} =\mathbb{E}\left[ \mathbf{1}\left\{ X_{k}\in \bigcap_{i\in %
	\left[ n\right] :i\neq k}B\left( X_{i}\right) ^{c}\right\} |\mathbf{X}%
^{\backslash k}\right] .
\end{equation*}%
The indicator of the event on the right hand side is a crude leave-one-out
estimate for the conditional missing mass. To reduce variance we average
this estimate over all $x_{k}$, which leads to the random variable%
\begin{equation*}
G\left( \mathbf{X}\right) =\frac{1}{n}\sum_{k=1}^{n}\mathbf{1}\left\{
X_{k}\in \bigcap_{i\in \left[ n\right] :i\neq k}B\left( X_{i}\right)
^{c}\right\} ,
\end{equation*}%
The random variable $G\left( \mathbf{X}\right) $ will also be called the 
\textit{Good-Turing estimator}, because this is what it reduces to in the
discrete case. It is the relative number of sample points, which are further
than $r$ from all other sample points.\bigskip

\begin{theorem}
	\label{Theorem Sourav Chatterjee} (Proof in Section \ref{Section Proof
		Good-Turing estimator}) Define%
	\begin{equation*}
	H\left( \mathbf{X},r\right) =\frac{1}{n}\sum_{k=1}^{n}\mu \left(
	\bigcap_{i\in \left[ n\right] :i\neq k}B\left( X_{i},r\right) ^{c}\right) 
	\text{.}
	\end{equation*}%
	Then
	
	(i) $\hat{M}\left( \mathbf{X},r\right) \leq H\left( \mathbf{X},r\right) \leq 
	\hat{M}\left( \mathbf{X},r\right) +1/n$
	
	(ii) $M\left( \mu ,n,r\right) \leq \mathbb{E}\left[ G\left( \mathbf{X}%
	,r\right) \right] \leq M\left( \mu ,n,r\right) +1/n$
	
	(iii) $\mathbb{V}\left[ G\left( \mathbf{X},r\right) -H\left( \mathbf{X}%
	,r\right) \right] \leq 3/n$
	
	(iv) $\left\Vert G\left( \mathbf{X},r\right) -\hat{M}\left( \mathbf{X}%
	,r\right) \right\Vert _{2}\leq \sqrt{7/n}.$
\end{theorem}

The random variable $H$ is an approximation of $\hat{M}$ adapted to the
Good-Turing estimator, since evidently $\mathbb{E}\left[ G\right] =\mathbb{E}%
\left[ H\right] $. The Conclusion (ii) is a simple extension of the bias
bound by \cite{good1953population} to the extended setting considered in
this paper. The variance bound (iii) and its tricky proof are due to Sourav
Chatterjee (private communication). It is unclear if higher-moment or
exponential bounds exist.

Parts (i) and (iii) of Theorem \ref{Theorem Sourav Chatterjee} in
combination with Chebychev's inequality show that, for $\delta >0$ with
probability at least $1-\delta $%
\begin{equation}
\left\vert \hat{M}\left( \mathbf{X},\epsilon \right) -G\left( \mathbf{X}%
,\epsilon \right) \right\vert \leq \frac{1}{n}+\sqrt{\frac{3}{n\delta }}.
\label{Chatterjee confidence bound}
\end{equation}

\subsection{A martingale estimator\label{Section Martingale estimator}}

The strong dependence of (\ref{Chatterjee confidence bound}) on the failure
probability $\delta $ makes it unsuited for the union bounds often used for
the purpose of model selection.

The conditional missing mass in some sense measures ignorance and it may in
some applications be more important to bound it above than below. This can
be done with the following estimator 
\begin{equation*}
T\left( \mathbf{X}\right) =\frac{1}{n}\sum_{k=1}^{n}\mathbf{1}\left\{
X_{k}\in \bigcap_{i<k}B\left( X_{i}\right) ^{c}\right\} .
\end{equation*}%
Notice the similarity to the Good-Turing estimator and $G\left( \mathbf{X}%
\right) \leq T\left( \mathbf{X}\right) $. It follows almost immediately from
the Hoeffding-Azuma Lemma \cite{McDiarmid98} that the difference $\hat{M}%
\left( \mathbf{X}\right) -T\left( \mathbf{X}\right) $\ has a sub-Gaussian
upper tail. But $T\left( \mathbf{X}\right) $ may have a large bias. To
reduce this we define for $m\in \left[ n\right] $ the random variable 
\begin{equation*}
T_{m}\left( \mathbf{X}\right) =\frac{1}{m}\sum_{k=n-m+1}^{n}1\left\{
X_{k}\in \bigcap_{i:i<k}B\left( X_{i},r\right) ^{c}\right\} .
\end{equation*}%
For $m=n$ this reduces to $T$. The estimator $T_{m}$ is related to, but not
the same as using $\mathbf{X}_{n-m+1}^{n}$ as a test set to estimate $\hat{M}%
\left( \mathbf{X}_{1}^{n-m}\right) $.\bigskip

\begin{theorem}
	\label{Theorem Martingale estimator} (Proof in Section \ref{Section Proof A
		martingale estimator}) For $t>0$ (i) $\mathbb{P}\left\{ \hat{M}\left( 
	\mathbf{X}\right) -T_{m}\left( \mathbf{X}\right) >t\right\} \leq
	e^{-mt^{2}/2}$.
	
	(ii) $\mathbb{P}\left\{ \hat{M}\left( \mathbf{X}\right) -2T_{m}\left( 
	\mathbf{X}\right) >t\right\} \leq \exp \left( -mt/\left( 4\left( e-2\right)
	\right) \right) .$
	
	(iii) For $m<n$, $\mathbb{E}\left[ T_{m}\left( \mathbf{X}\right) -\hat{M}%
	\left( \mathbf{X}\right) \right] \leq \ln \frac{n}{n-m}\leq m/\left(
	n-m\right) .$\bigskip
\end{theorem}

The two exponential tail bounds allow for some complicated union bounds
incurring only logarithmic penalties. For one example we may optimize the
bound in $m$. A union bound gives

\begin{corollary}
	For $\delta >0$%
	\begin{equation*}
	\mathbb{P}\left\{ \hat{M}\left( \mathbf{X}\right) >\min_{m\in \left[ n\right]
	}T_{m}\left( \mathbf{X}\right) +\sqrt{\frac{\ln \left( n/\delta \right) }{2m}%
}\right\} \leq \delta .
\end{equation*}
\end{corollary}

On the other hand we may be interested in uniform estimates of the minimal
conditional missing mass for sub-samples of a given size. For $S\subseteq
\left\{ 1,...,n\right\} $ denote with $\mathbf{X}^{S}$ the vector $\left(
X_{i}\right) _{i\in S}$. From Theorem \ref{Theorem Martingale estimator} and
a union bound we get

\begin{corollary}
	\label{Corollary Union bound bad set}For $m\in \left[ n\right] $ and $\delta
	>0$ 
	\begin{equation*}
	\mathbb{P}\left\{ \sup_{S:\left\vert S\right\vert =m}\hat{M}\left( \mathbf{X}%
	^{S}\right) -T\left( \mathbf{X}^{S}\right) >\sqrt{\frac{\min \left\{
			n-m,m\right\} \ln \left( n/\delta \right) }{m}}\right\} \leq \delta \text{.}
	\end{equation*}
\end{corollary}

The martingale estimators depend on the ordering of the indices. It is
possible to optimize the bound over a subset of permutations. The following
is an example of which takes advantage of $r$-nets of $\mathbf{X}$.

\begin{corollary}
	\label{Corollary bound missing mass with net}(Proof in Section \ref{Section
		Proof A martingale estimator}) For $m\in \left[ n\right] $ and $\delta >0$,
	with probability at least $1-\delta $ in $\mathbf{X}$, if there is an $r$%
	-net $\mathbf{Y\subset X}$, with $\left\vert \mathbf{Y}\right\vert =m$ then 
	\begin{equation*}
	\hat{M}\left( \mathbf{X},r\right) \leq \frac{m}{n}+\sqrt{\frac{m\ln \left(
			n/\delta \right) }{n}}
	\end{equation*}
\end{corollary}

\subsection{A negative result on the estimation of $M$ \label{Section
		negative results}}

In the discrete case it has been established that both $\hat{M}$ and $G$ are
exponentially concentrated on their expectations (\cite%
{mcallester2000convergence}, \cite{mcallester2003concentration}). From this
and the $1/n$-bias of $G$ it is immediate to obtain bounds on the estimation
error $G-\hat{M}$. In contrast to this Chatterjees's proof of Theorem \ref%
{Theorem Sourav Chatterjee} (iii) and the analysis of the martingale
estimator above adress the estimation error directly. This is in fact
necessary, because $\hat{M}$ and $G$ may have large variance.

For intuition into this fact let $\mu $ be a mixture of the uniform
distribution on $\mathbb{S}^{D-1}$, the unit sphere of $\mathbb{R}^{D}$
(with $D$ very large), and a small mass at the origin of $\mathbb{R}^{D}$.
Take $r\in \left( 1,\sqrt{2}\right) $. If $n\ll D$ and the origin is not in
the sample, the conditional missing mass will be nearly one, because the $%
X_{i}$ will be nearly mutually orthogonal and the spherical caps centered on
them have very small mass (this follows from isoperimetric theorems on the
sphere, see \cite{Ledoux91}, for example). By approximate orthogonality most
sample points will be alone in their respective balls, so $G$ will also be
large. But the entire support of the distribution is contained in the ball
about the origin, so, if the origin is in the sample, both $\hat{M}$ and $G$
drop to zero. If the probability of the origin being in the sample is $1/2$,
then the variance of $\hat{M}$ and $G$ is near the maximal value $1/4$.

Since this construction is possible for every sample-size $n$, no universal
and uniformly convergent estimator of $M\left( \mu ,n,r\right) $ exists in
the general case.

\begin{proposition}
	\label{Proposition Negative} (Proof in Section \ref{Section Proof A negative
		result}) Let $1<r<\sqrt{2}$. For every $\epsilon \in \left( 0,1\right) $ and 
	$n\in \mathbb{N}$ with $n\geq \ln \left( 4\right) /\epsilon $ there exists $%
	D\in \mathbb{N}$ and $\mu $ on $\mathbb{R}^{D}$ such that
	
	(i) for $\mathbf{X}\sim \mu ^{n}$, $\min \left\{ \mathbb{V}\left( \hat{M}%
	\left( \mathbf{X},r\right) \right) ,\mathbb{V}\left( G\left( \mathbf{X}%
	,r\right) \right) \right\} \geq \left( 1/4\right) -\epsilon $.
	
	(ii) Let $B$ be the event $\left\{ \forall i,j\text{ with }i\neq
	j,\left\Vert X_{i}-X_{j}\right\Vert >r\text{ and }\left\Vert
	X_{i}\right\Vert \leq 1\right\} $. Then $\mathbb{P}\left( B\right) \geq
	1/2-\epsilon $.
	
	(iii) For every $f:\mathcal{X}^{n}\rightarrow \mathbb{R}$ there exists $\mu
	^{\prime \prime }$ on $\mathbb{R}^{D}$ such that for $\mathbf{X}\sim \left(
	\mu ^{\prime \prime }\right) ^{n}$, we have 
	\begin{equation*}
	\mathbb{E}\left[ \left( f\left( \mathbf{X}\right) -M\left( \mu ^{\prime
		\prime },n,r\right) \right) ^{2}\right] \geq \left( 1-\epsilon \right)
	^{2}/16,
	\end{equation*}%
	and consequently $\left\Vert M-f\left( \mathbf{X}\right) \right\Vert
	_{L_{2}\left( \left( \mu ^{\prime \prime }\right) ^{n}\right) }\geq \left(
	1-\epsilon \right) /4$.
\end{proposition}

\subsection{Local separation\label{Section Condensation separation}}

It follows from Proposition \ref{Proposition Negative} that estimators of
the expected missing mass will only work well, if we can exclude a
construction as in the previous section. We can either rule it out a priori
by some constraint on the dimension, or, if we insist on dimension
independence, at least rule it out with high probability with the use of an
auxiliary statistic, which measures some intrinsic dimension of the
distribution.

For $r\geq 0$ and $k\in \mathbb{N}$ we say a sequence $S=\left(
x_{1},...,x_{k}\right) \in \mathcal{X}^{k}$ has the $r$-\textit{%
	local-separation} property, if

\begin{itemize}
	\item There exists $y\in \mathcal{X}$ such that $\forall i\in \left[ k\right]
	,$ $d\left( x_{i},y\right) \leq r$ (locality)
	
	\item For all $1\leq i<j\leq k$ we have $d\left( x_{i},x_{j}\right) >r$
	(separation)
\end{itemize}

So any sequence of points mutually separated by more than $r$ has this
property, if the intersection of the $r$-balls about them is non-empty. We
denote with $\Pi _{r}\subseteq \bigcup_{k\in 
	\mathbb{N}
}\mathcal{X}^{k}$ the set of all sequences $S$ having the $r$%
-local-separation property\textit{. }Define the function $h:\mathcal{X}%
^{n}\times \left[ 0,\infty \right) \rightarrow \mathbb{R}$ by 
\begin{equation*}
h\left( \mathbf{x},r\right) =\max \left\{ \left\vert S\right\vert
:S\subseteq \left( x_{1},...,x_{n}\right) \text{ such that }S\in \Pi
_{r}\right\} .
\end{equation*}%
$h\left( \mathbf{x},r\right) $ is the largest cardinality of a sub-sample
separated by more than $r$, but contained in some closed ball of radius $r$.

The next result shows that the random variable $h\left( \mathbf{X},r\right) $
controls concentration of $G$ and $\hat{M}$ about their expectations. Its
proof is somewhat complicated and uses some recent moment inequalities for
functions of independent variables.

\begin{theorem}
	\label{Theorem Main} (Proof in Section \ref{Section Proof Condensation and
		separation}) Under the conventions of Section \ref{Section Notation and
		conventions} let $n\geq 16$. Then%
	\begin{eqnarray}
	\mathbb{V}\left[ G\left( \mathbf{X},r\right) \right]  &\leq &\frac{2\left( 1+%
		\mathbb{E}\left[ h\left( \mathbf{X},r\right) \right] \right) }{n}  \notag \\
	\mathbb{V}\left[ \hat{M}\left( \mathbf{X},r\right) \right]  &\leq &\frac{2%
		\mathbb{E}\left[ h\left( \mathbf{X},r\right) \right] +4\left( e-2\right)
		\left( \ln n+1\right) }{n-1}.  \label{Variance bound M_hat}
	\end{eqnarray}%
	Furthermore, for any $t>0$,%
	\begin{eqnarray*}
		\mathbb{P}\left\{ \left\vert G\left( \mathbf{X},r\right) -\mathbb{E}\left[
		G\left( \mathbf{X},r\right) \right] \right\vert >12\sqrt{\frac{\left( 1+%
				\mathbb{E}\left[ h\left( \mathbf{X},r\right) \right] \right) t}{n}}+\frac{23t%
		}{\sqrt{n}}\right\}  &\leq &15e^{-t} \\
		\mathbb{P}\left\{ \left\vert \hat{M}\left( \mathbf{X},r\right) -\mathbb{E}%
		\left[ \hat{M}\left( \mathbf{X},r\right) \right] \right\vert >12\sqrt{\frac{%
				\mathbb{E}\left[ h\left( \mathbf{X},r\right) \right] t}{n}}+\frac{37t}{\sqrt{%
				n-1}}\right\}  &\leq &2ne^{-t}.
	\end{eqnarray*}%
	\bigskip 
\end{theorem}

Remarks:

1. \textbf{Tightness of variance bound.} Under the event $B$ described in
Proposition \ref{Proposition Negative} (ii) we have $h\left( \mathbf{X}%
,r\right) =n$. Since $\mathbb{P}\left( B\right) \geq 1/2-\epsilon $ we have $%
\mathbb{E}\left[ h\left( \mathbf{X},r\right) \right] /n\geq 1/4-\epsilon $.
With $\epsilon =1/8$ we get from Proposition \ref{Proposition Negative}%
\begin{equation*}
\frac{3}{64}\leq \frac{\mathbb{E}\left[ h\left( \mathbf{X},r\right) \right] 
}{8n}\leq \frac{1}{8}\leq \mathbb{V}\left( \hat{M}\left( \mathbf{X},r\right)
\right) ,
\end{equation*}%
so the variance bound (\ref{Variance bound M_hat}) is unimprovable up to a
constant factor and an additive term of $O\left( \ln \left( n\right)
/n\right) $.

2. \textbf{Finite dimensions.} In the discrete case, when $d\left(
x,y\right) =1$ $\iff $ $x\neq y$, and $r<1$ we always have $h\left( \mathbf{x%
},r\right) =1$. In one dimension $h\left( \mathbf{x}\right) $ is at most $2$%
, in $2$ dimensions it is at most $5$. In general we have the following
Proposition.

\begin{proposition}
	\label{Proposition packing number} (Proof in Section \ref{Section Proof
		Miscellaneous}) Let $\left( \mathbb{R}^{D},\left\Vert \text{.}\right\Vert
	\right) $ be a finite dimensional Banach space with closed unit ball $%
	\mathbb{B}$ and define the $1$-packing number of $\mathbb{B}$ as 
	\begin{equation*}
	\mathcal{P}\left( \mathbb{B},d_{\left\Vert .\right\Vert },1\right) :=\max
	\left\{ \left\vert S\right\vert :S\subset \mathbb{B}^{D},\forall x,y\in
	S,x\neq y\implies \left\Vert x-y\right\Vert >1\right\} .
	\end{equation*}%
	Let $r>0$. Then
	
	(i) for every vector $\mathbf{x}\in \left( \mathbb{R}^{D}\right) ^{n}$ we
	have $h\left( \mathbf{x},r\right) \leq \mathcal{P}\left( \mathbb{B}%
	,d_{\left\Vert .\right\Vert },1\right) \leq 8^{D}$.
	
	(ii) For the $2$-norm the bound improves to $3^{D}$.
	
	(iii) If $\mu $ has a positive density w.r.t. Lebesgue measure on $\mathbb{R}%
	^{D}$ and $\mathbf{X}_{1}^{n}\sim \mu ^{n}$ then $h\left( \mathbf{X}%
	_{1}^{n},r\right) \rightarrow \mathcal{P}\left( \mathbb{B},d_{\left\Vert
		.\right\Vert },1\right) $ almost surely as $n\rightarrow \infty $.
\end{proposition}

For any metric space $\left( \mathcal{X},d\right) $ with finite doubling
dimension DDim \cite{kontorovich2017nearest}\ we have $h\left( \mathbf{x}%
,r\right) \leq 2^{\text{DDim}}$, since the packing number at scale $r$ can
be bounded by the covering number for $r/2$ (\cite{vershynin2018high},
4.2.8). In summary: Theorem \ref{Theorem Main} guarantees exponential
concentration of $G$ and $\hat{M}$ on their expectations in all finite
dimensional metric spaces.

3. \textbf{Effective low dimensionality.} The worst-case bound for finite
dimensions is disappointing in its exponential dependence on the dimension.
But the random variable $h\left( \mathbf{X},r\right) $ depends on both the
underlying distribution and the scale $r$ and not on the dimension of the
ambient space. In the simplest case $\mu $ is supported on a low-dimensional
linear subspace, and the corresponding packing numbers can be used to bound $%
h\left( \mathbf{X},r\right) $. Linearity or smoothness however are not
necessary for $h\left( \mathbf{X},r\right) $ to be small, nor is
differentiability. There is a distribution $\mu $ in $L_{2}\left[ 0,\infty
\right) $ whose support is not totally bounded, nowhere smooth and not
contained in any finite dimensional subspace of $L_{2}\left[ 0,\infty
\right) $, but $h\left( \mathbf{X},r\right) \leq 5$ for any $r>0$ and $%
\mathbf{X}\sim \mu $ (Proposition \ref{Proposition example}).The assumption
of effective low-dimensionality is not unreasonable in practice, since the
generative processes underlying real-world distributions often have far
fewer degrees of freedom than the dimension of the ambient space where data
is presented, an observation which has given rise to the manifold hypothesis
(\cite{ma2012manifold}, \cite{fefferman2016testing}, \cite%
{berenfeld2021density}).

The next section addresses the question how the function $h\left( \mathbf{X}%
,r\right) $ can be estimated from the data.

\subsection{Concentration of $h$}

A subset $\Pi $ of the set of all sequences $\Pi \subseteq \bigcup_{k\in 
	\mathbb{N}
}\mathcal{X}^{k}$ is called \textit{hereditary}, if, whenever for $S=\left(
x_{1},...,x_{k}\right) \in \mathcal{X}^{k}$ we have $S\in \Pi $, then $%
S^{\prime }\in \Pi $ for every subsequence $S^{\prime }\subseteq S$. We
write $\Pi \left( S\right) $ for $S\in \Pi $. For example the property of a
sequence of real numbers to be non-decreasing is hereditary. Another example
is the local-separation property of a sequence of points $S=\left(
x_{1},...,x_{k}\right) $ in a space with symmetric distortion function, as
described in the previous section: if there exists $y$ such that $x_{i}\in
B\left( y,r\right) $ and $d\left( x_{i},x_{j}\right) >r$ for all $i\neq j$,
then the same will clearly hold for any subsequence of $S$.

The function $f_{\Pi }:\mathcal{X}^{k}\rightarrow \mathbb{N}_{0}$, which for 
$\mathbf{x}=\left( x_{1},...,x_{n}\right) $ gives the length $f_{\Pi }\left( 
\mathbf{x}\right) $ of the longest subsequence of $\mathbf{x}$, which has
hereditary property $\Pi $, is called the \textit{configuration function} of 
$\Pi $ (\cite{Boucheron13}, Section 3.3, see also \cite%
{talagrand1995concentration}, \cite{McDiarmid98} or \cite{boucheron2000sharp}%
). The function giving the length of the longest increasing subsequence in a
sequence of real numbers is such a configuration function, as is the
function $h\left( \mathbf{x},r\right) $ defined in the previous section.
Such functions have strong concentration properties. Here we quote Theorem
6.12 in \cite{Boucheron13}).

\begin{theorem}
	\label{Theorem SelfBound} If $\mathbf{X}=\left( X_{1},...,X_{n}\right) $ is
	a vector of independent variables in $\mathcal{X}$ and $f_{\Pi }:\mathcal{X}%
	^{n}\rightarrow \mathbb{R}$ is the configuration function corresponding to
	the hereditary property $\Pi $ above then
	
	(i) for every $t>0$%
	\begin{equation*}
	\mathbb{P}\left\{ f_{\Pi }\left( \mathbf{X}\right) -\mathbb{E}\left[ f_{\Pi
	}\left( \mathbf{X}\right) \right] >t\right\} \leq \exp \left( \frac{-t^{2}}{2%
	\mathbb{E}\left[ f_{\Pi }\left( \mathbf{X}\right) \right] +2t/3}\right) ,
\end{equation*}

(ii) and for every $0<t\leq \mathbb{E}\left[ f_{\Pi }\left( \mathbf{X}%
\right) \right] $%
\begin{equation*}
\mathbb{P}\left\{ \mathbb{E}\left[ f_{\Pi }\left( \mathbf{X}\right) \right]
-f_{\Pi }\left( \mathbf{X}\right) >t\right\} \leq \exp \left( \frac{-t^{2}}{2%
	\mathbb{E}\left[ f_{\Pi }\left( \mathbf{X}\right) \right] }\right) .
\end{equation*}
\end{theorem}

We can immediately substitute $h\left( \cdot ,r\right) $ for $f_{\Pi }$. For
our purpose the most important consequences are summarized in the following.

\begin{corollary}
	\label{Lemma bound EH by sample} (Proof in Section \ref{Section Proof
		Miscellaneous}) For $t>0$
	
	(i) $\mathbb{P}\left\{ \sqrt{\mathbb{E}\left[ h\left( \mathbf{X},r\right) %
		\right] }\leq \sqrt{h\left( \mathbf{X},r\right) }+\sqrt{2t}\right\} \geq
	1-e^{-t}$
	
	(ii) $\mathbb{P}\left\{ h\left( \mathbf{X},r\right) -2\mathbb{E}\left[
	h\left( \mathbf{X},r\right) \right] >t\right\} \leq e^{-6t/7}.$
\end{corollary}

Part (i) means that, if we are able to compute $h\left( \mathbf{X}\right) $,
then $\mathbb{E}\left[ h\left( \mathbf{X}\right) \right] $ can be estimated
with high probability from the sample. Consequently the bounds in Theorem %
\ref{Theorem Main} can be independent of assumptions on the distribution $%
\mu $ and determined with high probability by the observed data $\mathbf{X}$%
, as can be seen by combining part (i) with the eponential inequalities of
Theorem \ref{Theorem Main} in a union bound.

Part (ii) gives a sub-exponential bound in the other direction, which will
be instrumental in the proof of Theorem \ref{Theorem Main}.

At this point we have no efficient algorithm to compute $h\left( \mathbf{x}%
,r\right) $, if this number is large, most likely this problem is NP-hard.
But for our bounds it might be sufficient to determine if $h\left( \mathbf{x}%
,r\right) \geq h_{0}$ for some fixed value $h_{0}$ and to compute it
otherwise. In the euclidean space $\mathbb{R}^{D}$ one could execute an
algorithm for the minimum enclosing ball problem (e.g. \cite{yildirim2008two}%
) of $O\left( nD\right) $ on the $O\left( n^{h_{0}}\right) $ candidate
subsequences of size $h_{0}$, which would take polynomial execution time $%
O\left( n^{h_{0}+1}D\right) $. The generation of candidate subsequences
could be further accelerated as they have to satisfy $r<d\left(
x_{i},x_{j}\right) \leq 2r$.

If we relax the locality condition to $d\left( x_{i},x_{j}\right) \leq 2r$
then the cumbersome minimum-enclosing-ball problem can be avoided, and
computation of the relaxed statistic is equivalent to the $h_{0}$-clique
problem for the graph with $n$ vertices and edges whenever $r<d\left(
x_{i},x_{j}\right) \leq 2r$. In this case an efficient algorithm is given in 
\cite{vassilevska2009efficient}. In any case the computation of $h\left( 
\mathbf{x},r\right) $, or a good upper bound thereof, remains an interesting
problem for further research.

\section{Applications\label{Section applications}}

Since $h\left( \mathbf{x}\right) =1$ in the discrete case, many of the
applications of the discrete case are covered by Theorem \ref{Theorem Main},
albeit with larger constants. In this section we sketch a few applications
not covered by the classical results.

\subsection{Anomaly detection}

Kontorovich et al \cite{kontorovich2011metric} propose a method of anomaly
detection, where they assume that the metric space $(\mathcal{X},d)$ is
partitioned into disjoint sets corresponding to "normal" and "anomalous"
ones, being separated by some minimal separation distance $\gamma $, so that 
$d\left( x,y\right) >\gamma $ for every pair of a normal point $x$ and an
anomalous point $y$. Training data $\mathbf{X}$ is drawn from an unknown
distribution $\mu $ supported on the normal points. If the separation
distance $\gamma $ is known, the simplest rule for anomaly detection is the
proximity classifier, which decides a point $y$ to be anomalous iff $d\left(
y,X_{i}\right) >\gamma $ for all $X_{i}$ in the sample. Then the "false
alarm rate" (the probability that a normal point is labeled as anomalous) is
the conditional missing mass $\hat{M}\left( \mathbf{X}\right) $ and a
data-dependent bound may be given either with the Good Turing estimator or
any of the estimators in Section \ref{Section Martingale estimator}.

\subsection{Nearest neighbor coding}

Given a sample $\mathbf{X}\sim \mu ^{n}$ we encode every point $x\in 
\mathcal{X}$ by the index of the nearest neighbor in the sample, that is by $%
i\left( x\right) =\arg \min_{i\in \left[ n\right] }d\left( x,X_{i}\right) $.
Given a code $i\left( x\right) $ we reconstruct the point $x$ as $X_{i\left(
	x\right) }$ and incur a reconstruction error $d\left( x,X_{i\left( x\right)
}\right) $. Then the probability that the reconstruction error exceeds some
specified accuracy $\epsilon $ is clearly $\mu \left( X:d\left( X,X_{i\left(
	X\right) }\right) >\epsilon \right) =\hat{M}\left( \mathbf{X},\epsilon
\right) $. Using Theorem \ref{Theorem Sourav Chatterjee}, it may be
estimated by the Good-Turing estimator $G$ as 
\begin{equation*}
\left\vert \hat{M}\left( \mathbf{X},\epsilon \right) -G\left( \mathbf{X}%
,\epsilon \right) \right\vert \leq \sqrt{\frac{3}{n\delta }}
\end{equation*}%
with probability at least $1-\delta $ in the sample $\mathbf{X}$.
Alternatively we may upper bound the reconstruction error with probability
at least $1-\delta $ as%
\begin{eqnarray*}
	\hat{M}\left( \mathbf{X},\epsilon \right)  &\leq &T_{n}\left( \mathbf{X}%
	\text{,}\epsilon \right) +\sqrt{\frac{\ln \left( 1/\delta \right) }{2n}}%
	\text{ or} \\
	\hat{M}\left( \mathbf{X},\epsilon \right)  &\leq &\min_{m\in \left[ n\right]
	}T_{m}\left( \mathbf{X},\epsilon \right) +\sqrt{\frac{\ln \left( n/\delta
		\right) }{2m}},
\end{eqnarray*}%
using Theorem \ref{Theorem Martingale estimator} or Corollary \ref{Corollary
	Union bound bad set} (i). If the distortion function is bounded, say $%
d\left( x,y\right) \leq \Delta $, then the expected reconstruction error can
be bounded by $\mathbb{E}\left[ d\left( X,X_{i\left( X\right) }\right) |%
\mathbf{X}\right] \leq \Delta \hat{M}\left( \mathbf{X},\epsilon \right)
+\epsilon $ and estimated in the same way.

In high dimensions these estimates are, albeit correct, manifestly sample
dependent and not necessarily reproducible. It follows from Proposition \ref%
{Proposition Negative} that two samples $\mathbf{X}$ and $\mathbf{X}^{\prime
}$ may differ in a single point with $\hat{M}\left( \mathbf{X},\epsilon
\right) =0$ and $\hat{M}\left( \mathbf{X}^{\prime },\epsilon \right) $
arbitrarily close to $1$. Theorem \ref{Theorem Main} then gives exponential
guarantees of reproducibility in terms of the quantity $\mathbb{E}\left[
h\left( \mathbf{X},\epsilon \right) \right] $, which depends on the
intrinsic dimension of $\mu $, the scale $\epsilon $ and the sample size $n$.

If we are content with any reconstruction error smaller than $\epsilon $,
the coding scheme above is redundant and inefficient, whenever sample points
cluster at scales much smaller than $\epsilon $. In this case we can
construct an $\epsilon /2$-net $\mathbf{Y}$ of $\mathbf{X}$ (a maximal $%
\epsilon /2$-separated subsequence of $\mathbf{X}$) and encode with nearest
neighbors of $\mathbf{Y}$. Since every point in $\mathbf{X}$ is within $%
\epsilon /2$ from some point of $\mathbf{Y}$, the probability that the
reconstruction error of this coding scheme exceeds $\epsilon $ is then
bounded by $\hat{M}\left( \mathbf{X},\epsilon /2\right) $ and can again be
estimated as above.

Similar coding schemes, which use sub-sampled nets, underlie the
nearest-neighbor sample-compression classification algorithm developed in 
\cite{kontorovich2017nearest}. The recent paper \cite{hanneke2020universal}
proves that a minor modification of this algorithm, called OPTINET, is
universally Bayes consistent in all essentially separable metric spaces. In
this proof a bound on the conditional missing mass in the general setting of
metric spaces, as defined by training sample and input marginal, is
essential. The authors use a partitioning scheme as in the proof of
Proposition \ref{Proposition Convergence} to reduce the estimation problem
to the discrete case, which is overly pessimistic. This does no harm
however, as the results to be proven are asymptotic. The same method is used
in the recent paper \cite{cohen2022learning}. In the next section similar
ideas are used together with the results in this paper to obtain finite
sample bounds.

\subsection{The Wasserstein distance to the empirical distribution}

Suppose $\left( \mathcal{X},d\right) $ is a metric space. The Wasserstein
distance $W_{1}\left( \mu ,\nu \right) $ on probability measures $\mu $ and $%
\nu $ is usually defined in terms of couplings or optimal transport. By the
Kantorovich-Rubinstein Theorem it can be equivalently defined as 
\begin{equation}
W_{1}\left( \mu ,\nu \right) =\sup_{\left\Vert f\right\Vert _{Lip}=1}\int
f\left( d\mu -d\nu \right) ,  \label{Kantorovich Rubinstein}
\end{equation}%
where $\left\Vert f\right\Vert _{Lip}$ is the usual Lipschitz seminorm. One
quantity which has attracted attention is $W_{1}\left( \mu ,\hat{\mu}\right) 
$, where $\hat{\mu}$ is the empirical distribution%
\begin{equation*}
\hat{\mu}=\frac{1}{n}\sum_{i=1}^{n}\delta _{X_{i}}\text{ for }X=\left(
X_{1},...,X_{n}\right) \sim \mu ^{n}\text{.}
\end{equation*}%
Dudley \cite{dudley1969speed} has shown that $W_{1}\left( \mu ,\hat{\mu}%
\right) \approx n^{-1/D}$ if $\mu $ is compactly supported on $\mathbb{R}^{D}
$. This result has since been refined by several authors. Notably Weed and
Bach \cite{weed2019sharp} have sharpened and generalized this by moving to
general bounded metric spaces and replacing $D$ by an intrinsic dimension of
the probability measure $\mu $. In this section we give a simple and purely
empirical bound on $W_{1}\left( \mu ,\hat{\mu}\right) $.

First of all note that 
\begin{equation*}
W_{1}\left( \mu ,\hat{\mu}\right) \geq r\hat{M}\left( \mathbf{X},r\right) 
\text{ for every }r>0\text{.}
\end{equation*}%
This is obvious from the optimal transport interpretation, as the missing
mass has to be moved at least a distance $r$ to arrive at the sample.
Formally the supremum in the definition above is then witnessed by the
Lipschitz function $x\mapsto \min_{i\in \left[ n\right] }d\left(
x,X_{i}\right) $.

The estimate in the other direction is more complicated, because we have to
control the error within the envelope $\bigcup_{i}B\left( X_{i},r\right) $.
For this we require an $r$-net of the sample, and the analysis we provide is
strongly inspired by the nearest-neighbor sample-compression methods
developed in \cite{hanneke2020universal} or \cite{cohen2022learning}$.$

\begin{theorem}
	\label{Theorem Wasserstein}(Proof in Section \ref{Section Proof Wasserstein}%
	) Let $\left( \mathcal{X},d\right) $ be a complete, separable metric space
	with diameter $1$ and Borel probability measure $\mu $. For $\delta >0$,
	with probability at least $1-\delta $ in $\mathbf{X}\sim \mu ^{n}$, if there
	exists an $r$-net $\mathbf{Y}\subset \mathbf{X}$ with cardinality $m$, $%
	m\leq \left( n-3\right) /2$, then%
	\begin{equation*}
	W_{1}\left( \mu ,\hat{\mu}\right) \leq \hat{M}\left( \mathbf{X},r\right)
	+3r+2\sqrt{\frac{m}{n-m}}\left( 1+\sqrt{\ln \left( n/\delta \right) }\right) 
	\end{equation*}%
	or%
	\begin{equation*}
	W_{1}\left( \mu ,\hat{\mu}\right) \leq 3r+3\sqrt{\frac{m}{n-m}}\left( 1+%
	\sqrt{\ln \left( 2n/\delta \right) }\right) .
	\end{equation*}
\end{theorem}

This bound can be optimized by generating the $r$-nets of $\mathbf{X}$ using
a farthest-first traversal algorithm. The price of data-dependence and
simplicity is that we don't quite recover the $n^{-1/D}$-bound for$\mu $
compactly supported on $\mathbb{R}^{D}$. But in that case we can always find
a $n^{-1/\left( D+2\right) }$-net of cardinality $O\left( n^{D/\left(
	D+2\right) }\right) $ for any sample from such a distribution, and Theorem %
\ref{Theorem Wasserstein} then gives a bound of $O\left( \ln \left( n\right)
/n^{-1/\left( D+2\right) }\right) $.

\subsection{Elementary learning bounds for $\protect\beta $-smooth functions}

We give a very easy data-dependent learning bound involving a rather large
hypothesis class, where the conditional missing mass controls generalization
as a data dependent complexity measure. It shows how learning is possible
for "easy" data, even if standard complexity measures on the hypothesis
class fail.

Let $\mathcal{F}$ be a loss-class on $\left( \mathcal{X},d\right) $. By this
we mean that $\mathcal{F}$ is the set of functions obtained from composing
the hypothesis functions with a fixed, non-negative loss function. Draw a
training sample $\mathbf{X\sim }\mu ^{n}$ and let $\mathcal{F}_{\mathbf{X}}$
be the class of loss functions which have zero empirical error, that is%
\begin{equation*}
\mathcal{F}_{\mathbf{X}}=\left\{ f\in \mathcal{F}:\forall i\in \left[ n%
\right] ,f\left( X_{i}\right) =0\right\} .
\end{equation*}%
Given a test variable $X\sim \mu $, which is independent of $\mathbf{X}$,
and a tolerance parameter $s$ we define an error functional by%
\begin{equation*}
\mathcal{R}\left( \mathbf{X,}s\right) =\mathbb{P}\left\{ \exists f\in 
\mathcal{F}_{\mathbf{X}},\text{ }f\left( X\right) >s|~\mathbf{X}\right\} .
\end{equation*}%
As it stands the loss may be arbitrarily large on the bad event, whose
probability we want to bound, but on the good event it is uniformly bounded.
This is different from conventional risk bounds, which would involve the
expectation \thinspace $\mathbb{E}\left[ f\left( X\right) \right] $. If the
loss functions were uniformly bounded, we could convert a bound on $\mathcal{%
	R}\left( \mathbf{X,}s\right) $ into a risk bound of the form%
\begin{equation*}
\forall f\in \mathcal{F}_{\mathbf{X}},\mathbb{E}\left[ f\left( X\right) %
\right] \leq s+\mathcal{R}\left( \mathbf{X,}s\right) \sup_{f\in \mathcal{F}%
}\left\Vert f\right\Vert _{\infty }.
\end{equation*}

Now take $\left( \mathcal{X},d\right) $ to be a Hilbert-space and assume
that the functions in $\mathcal{F}$ are $\beta $-smooth, which means that
their gradients $f^{\prime }$ are $\beta $-Lipschitz. Such a condition is
standard for optimization algorithms involving gradient descent. For $\beta $%
-smooth functions the fundamental theorem of calculus implies the inequality%
\begin{equation*}
f\left( x\right) -f\left( y\right) \leq \left\langle f^{\prime }\left(
y\right) ,x-y\right\rangle +\frac{\beta }{2}\left\Vert x-y\right\Vert ^{2}.
\end{equation*}%
Now if $f\in \mathcal{F}_{\mathbf{X}}$ then $f^{\prime }\left( X_{i}\right)
=f\left( X_{i}\right) =0$, since $f$ is non-negative, differentiable and
vanishes at $X_{i}$. Therefore%
\begin{eqnarray*}
	\mathcal{R}\left( \mathbf{X,}s\right) &=&\mathbb{P}\left\{ \exists f\in 
	\mathcal{F}_{\mathbf{X}},\text{ }f\left( X\right) >s|~\mathbf{X}\right\} \\
	&\leq &\Pr \left\{ \exists f\in \mathcal{F}:\forall i\in \left[ n\right] ,%
	\text{ }f\left( X\right) -f\left( X_{i}\right) >s~|\mathbf{X}\right\} \\
	&\leq &\Pr \left\{ \forall i\in \left[ n\right] ,~\frac{\beta }{2}\left\Vert
	X-X_{i}\right\Vert ^{2}>s|~\mathbf{X}\right\} =\hat{M}\left( \mathbf{X},%
	\sqrt{\frac{2s}{\beta }}\right) .
\end{eqnarray*}%
$\hat{M}\left( \mathbf{X},\sqrt{2s/\beta }\right) $ can be estimated by the
methods described. Using Corollary \ref{Corollary Union bound bad set} it is
also possible to allow a certain fraction of errors, where $f\left(
X_{i}\right) >0$.

\section{Proofs}

For the reader's convenience the various theorems and propositions are
restated.

\subsection{The Good-Turing estimator\label{Section Proof Good-Turing
		estimator}}

\begin{theorem}[\textbf{= Theorem \protect\ref{Theorem Sourav Chatterjee}}]
	\label{Theorem Sourav Chatterjee Restatement}Define%
	\begin{equation*}
	H\left( \mathbf{X},r\right) =\frac{1}{n}\sum_{k=1}^{n}\mu \left(
	\bigcap_{i\in \left[ n\right] :i\neq k}B\left( X_{i},r\right) ^{c}\right) 
	\text{.}
	\end{equation*}%
	Then
	
	(i) $\hat{M}\left( \mathbf{X},r\right) \leq H\left( \mathbf{X},r\right) \leq 
	\hat{M}\left( \mathbf{X},r\right) +1/n$
	
	(ii) $M\left( \mu ,n,r\right) \leq \mathbb{E}\left[ G\left( \mathbf{X}%
	,r\right) \right] \leq M\left( \mu ,n,r\right) +1/n$
	
	(iii) $\mathbb{V}\left[ G\left( \mathbf{X},r\right) -H\left( \mathbf{X}%
	,r\right) \right] \leq 3/n$
	
	(iv) $\left\Vert G\left( \mathbf{X},r\right) -\hat{M}\left( \mathbf{X}%
	,r\right) \right\Vert _{2}\leq \sqrt{7/n}.$
\end{theorem}

\begin{proof}
	We introduce a shorthand notation for some random subsets of $\mathcal{X}$.
	For $i\in \left[ n\right] $ we write $B_{i}=B\left( X_{i},r\right) $ and for 
	$i,j\in \left[ n\right] $, $i\neq j$%
	\begin{equation*}
	U=\bigcup_{k\in \left[ n\right] }B_{k}\text{ , }U_{i}=\bigcup_{k\in \left[ n%
		\right] \backslash \left\{ i\right\} }B_{k}\text{ and }U_{ij}=\bigcup_{k\in %
		\left[ n\right] \backslash \left\{ i,j\right\} }B_{k}.
	\end{equation*}%
	Then $\hat{M}^{\perp }=\mu \left( U\right) $, $G^{\perp }=\left( 1/n\right)
	\sum_{i\in \left[ n\right] }\mathbf{1}\left\{ X_{i}\in U_{i}\right\} $ and $%
	H^{\perp }=\left( 1/n\right) \sum_{i\in \left[ n\right] }\mu \left(
	U_{i}\right) $. Since $U_{i}$ is independent of $X_{i}$ we have $\mathbb{E}%
	\left[ \mathbf{1}\left\{ X_{i}\in U_{i}\right\} |\mathbf{X}^{\backslash i}%
	\right] =\mu \left( U_{i}\right) $, so that $\mathbb{E}\left[ G^{\perp }%
	\right] =\mathbb{E}\left[ H^{\perp }\right] $. Also $\mathbb{E}\left[ 
	\mathbf{1}\left\{ X_{i}\in U_{ij}\right\} |\mathbf{X}^{\backslash i}\right]
	=\mu \left( U_{ij}\right) =\mathbb{E}\left[ \mathbf{1}\left\{ X_{j}\in
	U_{ij}\right\} |\mathbf{X}^{\backslash j}\right] $. Note that 
	\begin{equation}
	U_{ij}\subseteq U_{i}\subseteq U\text{, }U\backslash U_{i}=B_{i}\backslash
	U_{i}\text{, }U_{j}\backslash U_{ij}=B_{i}\backslash U_{ij}.
	\label{inclusions}
	\end{equation}%
	The collection of sets $\left\{ B_{i}\backslash U_{i}\right\} _{i\in \left[ n%
		\right] }$ and for fixed $j\in \left[ n\right] $ the collection $\left\{
	B_{i}\backslash U_{ij}\right\} _{i\in \left[ n\right] \backslash \left\{
		j\right\} }$ and the collection of events $\left( \left\{ X_{j}\in
	B_{i}\backslash U_{ij}\right\} \right) _{i\in \left[ n\right] \backslash
		\left\{ j\right\} }$ are all disjoint.
	
	We address the bias first. By (\ref{inclusions}) $H^{\perp }\leq \hat{M}%
	^{\perp }$ and 
	\begin{eqnarray*}
		\hat{M}^{\perp }-H^{\perp } &=&\frac{1}{n}\sum_{i\in \left[ n\right] }\left(
		\mu \left( U\right) -\mu \left( U_{i}\right) \right) =\frac{1}{n}\sum_{i\in %
			\left[ n\right] }\mu \left( U\backslash U_{i}\right) \\
		&=&\frac{1}{n}\sum_{i\in \left[ n\right] }\mu \left( B_{i}\backslash
		U_{i}\right) =\frac{1}{n}\mu \left( \bigcup_{i\in \left[ n\right]
		}B_{i}\backslash U_{i}\right) \leq \frac{1}{n}.
	\end{eqnarray*}%
	This gives (i). Taking the expectation gives (ii).
	
	We now come to Chatterjee's variance bound. Fix $j\in \left[ n\right] $ for
	the moment. For $i\in \left[ n\right] \backslash \left\{ j\right\} $ we have 
	$\mathbb{E}\left[ \mathbf{1}\left\{ X_{j}\in U_{j}\right\} -\mathbf{1}%
	\left\{ X_{j}\in U_{ij}\right\} |\mathbf{X}^{\backslash j}\right] =\mu
	\left( U_{j}\right) -\mu \left( U_{ij}\right) $. In view of the inclusions
	in (\ref{inclusions}) the unconditional expectation gives 
	\begin{eqnarray}
	\mathbb{E}\left[ \left\vert \mu \left( U_{j}\right) -\mu \left(
	U_{ji}\right) \right\vert \right]  &=&\mathbb{E}\left[ \mu \left(
	U_{j}\right) -\mu \left( U_{ji}\right) \right]   \notag \\
	&=&\mathbb{E}\left[ \left\vert \mathbf{1}\left\{ X_{j}\in U_{j}\right\} -%
	\mathbf{1}\left\{ X_{j}\in U_{ij}\right\} \right\vert \right] =\mathbb{E}%
	\left[ \mathbf{1}\left\{ X_{j}\in U_{j}\right\} -\mathbf{1}\left\{ X_{j}\in
	U_{ij}\right\} \right]   \notag \\
	&=&\mathbb{E}\left[ \mathbf{1}\left\{ X_{j}\in U_{j}/U_{ij}\right\} \right] =%
	\mathbb{P}\left( \left\{ X_{j}\in B_{i}/U_{ij}\right\} \right) .
	\label{crucial estimate}
	\end{eqnarray}%
	Since $X_{j}$ and $U_{ij}$ are independent of $X_{i}$ 
	\begin{align*}
	& \mathbb{E}\left[ \left( \mathbf{1}\left\{ X_{i}\in U_{i}\right\} -\mu
	\left( U_{i}\right) \right) \left( \mathbf{1}\left\{ X_{j}\in U_{ij}\right\}
	-\mu \left( U_{ij}\right) \right) \right]  \\
	& =\mathbb{E}\left[ \mathbb{E}\left[ \mathbf{1}\left\{ X_{i}\in
	U_{i}\right\} -\mu \left( U_{i}\right) |\mathbf{X}^{\backslash i}\right]
	\left( \mathbf{1}\left\{ X_{j}\in U_{ij}\right\} -\mu \left( U_{ij}\right)
	\right) \right]  \\
	& =0.
	\end{align*}%
	On the other hand $\left\vert \mathbf{1}\left\{ X_{i}\in U_{i}\right\} -\mu
	\left( U_{i}\right) \right\vert \leq 1$, so that, for any $i\neq j$,%
	\begin{align*}
	& \mathbb{E}\left[ \left( \mathbf{1}\left\{ X_{i}\in U_{i}\right\} -\mu
	\left( U_{i}\right) \right) \left( \mathbf{1}\left\{ X_{j}\in U_{j}\right\}
	-\mu \left( U_{j}\right) \right) \right]  \\
	& =\mathbb{E}\left[ \left( \mathbf{1}\left\{ X_{i}\in U_{i}\right\} -\mu
	\left( U_{i}\right) \right) \left( \left( \mathbf{1}\left\{ X_{j}\in
	U_{j}\right\} -\mu \left( U_{j}\right) \right) -\left( \mathbf{1}\left\{
	X_{j}\in U_{ij}\right\} -\mu \left( U_{ij}\right) \right) \right) \right]  \\
	& \leq \mathbb{E}\left[ \left\vert \left( \mathbf{1}\left\{ X_{j}\in
	U_{j}\right\} -\mu \left( U_{j}\right) \right) -\left( \mathbf{1}\left\{
	X_{j}\in U_{ij}\right\} -\mu \left( U_{ij}\right) \right) \right\vert \right]
	\\
	& \leq \mathbb{E}\left[ \left\vert \mathbf{1}\left\{ X_{j}\in U_{j}\right\}
	-1\left\{ X_{j}\in U_{ij}\right\} \right\vert \right] +\mathbb{E}\left[
	\left\vert \mu \left( U_{j}\right) -\mu \left( U_{ij}\right) \right\vert %
	\right]  \\
	& =2\mathbb{P}\left\{ X_{j}\in B_{i}/U_{ij}\right\} \text{,}
	\end{align*}%
	where the last equality follows from (\ref{crucial estimate}). Thus%
	\begin{align*}
	& \mathbb{V}\left[ G-H\right] =\mathbb{V}\left[ G^{\perp }-H^{\perp }\right] 
	\\
	& =\frac{1}{n^{2}}\sum_{i}\mathbb{E}\left[ \left( \mathbf{1}\left\{ X_{i}\in
	U_{i}\right\} -\mu \left( U_{i}\right) \right) ^{2}\right]  \\
	& \text{ \ \ \ \ \ \ \ }+\frac{1}{n^{2}}\sum_{j}\sum_{i:i\neq j}\mathbb{E}%
	\left[ \left( \mathbf{1}\left\{ X_{i}\in U_{i}\right\} -\mu \left(
	U_{i}\right) \right) \left( \mathbf{1}\left\{ X_{j}\in U_{j}\right\} -\mu
	\left( U_{j}\right) \right) \right]  \\
	& \leq \frac{1}{n}+\frac{2}{n^{2}}\sum_{j}\left( \sum_{i:i\neq j}\mathbb{P}%
	\left( \left\{ X_{j}\in B_{i}\backslash U_{ij}\right\} \right) \right)  \\
	& =\frac{1}{n}+\frac{2}{n^{2}}\sum_{j}\mathbb{P}\left( \bigcup_{i:i\neq
		j}\left\{ X_{j}\in B_{i}\backslash U_{ij}\right\} \right) \text{ (*)} \\
	& \leq \frac{1}{n}+\frac{2}{n}=\frac{3}{n}.
	\end{align*}%
	The identity in (*) holds, since the events $\left\{ X_{j}\in
	B_{i}\backslash U_{ij}\right\} $ in the sum over $i:i\neq j$ in the line
	before are disjoint. This proves (iii), and together with (i) and $\left(
	a+b\right) ^{2}\leq 2a^{2}+2b^{2}$ it shows that 
	\begin{equation*}
	\mathbb{E}\left[ \left\vert G-\hat{M}\right\vert ^{2}\right] \leq \mathbb{E}%
	\left[ \left( \left( G-H\right) +\left( H-\hat{M}\right) \right) ^{2}\right]
	\leq \frac{6}{n}+\frac{2}{n^{2}}\leq \frac{7}{n}.
	\end{equation*}
\end{proof}

\subsection{A martingale estimator\label{Section Proof A martingale
		estimator}}

\begin{theorem}[\textbf{= Theorem \protect\ref{Theorem Martingale estimator}}%
	]
	\label{Theorem Martingale estimator Restatement}For $t>0$ (i) $\mathbb{P}%
	\left\{ \hat{M}\left( \mathbf{X}\right) -T_{m}\left( \mathbf{X}\right)
	>t\right\} \leq e^{-mt^{2}/2}$.
	
	(ii) $\mathbb{P}\left\{ \hat{M}\left( \mathbf{X}\right) -2T_{m}\left( 
	\mathbf{X}\right) >t\right\} \leq \exp \left( -mt/\left( 4\left( e-2\right)
	\right) \right) .$
	
	(iii) For $m<n$, $\mathbb{E}\left[ T_{m}\left( \mathbf{X}\right) -\hat{M}%
	\left( \mathbf{X}\right) \right] \leq \ln \frac{n}{n-m}\leq m/\left(
	n-m\right) .$\bigskip
\end{theorem}

For the proof of the relative bound (ii) (and also of Lemma \ref{Lemma Bound
	Wk against h} below) we need the following lemma, which is a minor
modification and application of Theorem 1 of (\cite%
{beygelzimer2011contextual}.\bigskip 

\begin{lemma}
	\label{Lemma Martingale}Let $R_{1},...,R_{n}$ be random variables $0\leq
	R_{j}\leq 1$ and let $\mathcal{F}_{j}$ be the $\sigma $-algebra generated by 
	$R_{1},...,R_{j}$. Let $V=\frac{1}{n}\sum_{j}R_{j}$, $F=\frac{1}{n}\sum_{j}%
	\mathbb{E}\left[ R_{j}|\mathcal{F}_{j-1}\right] $. Then%
	\begin{equation*}
	1\geq \mathbb{E}\left[ \exp \left( \left( \frac{n}{4\left( e-2\right) }%
	\right) \left( F-2V\right) \right) \right] .
	\end{equation*}
\end{lemma}

\begin{proof}
	Let $Y_{j}:=\frac{1}{n}\left( \mathbb{E}\left[ R_{j}|\mathcal{F}_{j-1}\right]
	-R_{j}\right) $, so $\mathbb{E}\left[ Y_{j}|\mathcal{F}_{j-1}\right] =0$.%
	\newline
	Then $\mathbb{E}\left[ Y_{j}^{2}|\mathcal{F}_{j-1}\right] =\left(
	1/n^{2}\right) \left( \mathbb{E}\left[ R_{j}^{2}|\mathcal{F}_{j-1}\right] -%
	\mathbb{E}\left[ R_{j}|\mathcal{F}_{j-1}\right] ^{2}\right) \leq \left(
	1/n\right) ^{2}\mathbb{E}\left[ R_{j}|\mathcal{F}_{j-1}\right] $, since $%
	0\leq R_{j}\leq 1$. For $\beta <n$ we have, using $e^{x}\leq 1+x+\left(
	e-2\right) x^{2}$ for $x\leq 1$,%
	\begin{eqnarray*}
		\mathbb{E}\left[ e^{\beta Y_{j}}|\mathcal{F}_{j-1}\right] &\leq &\mathbb{E}%
		\left[ 1+\beta Y_{j}+\left( e-2\right) \beta ^{2}Y_{j}^{2}|\mathcal{F}_{j-1}%
		\right] \\
		&=&1+\left( e-2\right) \beta ^{2}\mathbb{E}\left[ Y_{j}^{2}|\mathcal{F}_{j-1}%
		\right] \\
		&\leq &\exp \left( \left( e-2\right) \beta ^{2}\mathbb{E}\left[ Y_{j}^{2}|%
		\mathcal{F}_{j-1}\right] \right) \\
		&\leq &\exp \left( \left( e-2\right) \left( \frac{\beta }{n}\right) ^{2}%
		\mathbb{E}\left[ R_{j}|\mathcal{F}_{j-1}\right] \right) ,
	\end{eqnarray*}%
	where we also used $1+x\leq e^{x}$. Defining $Z_{0}=1$ and for $j\geq 1$ 
	\begin{equation*}
	Z_{j}=Z_{j-1}\exp \left( \beta Y_{j}-\left( e-2\right) \left( \frac{\beta }{n%
	}\right) ^{2}\mathbb{E}\left[ R_{j}|\mathcal{F}_{j-1}\right] \right)
	\end{equation*}%
	then 
	\begin{equation*}
	\mathbb{E}\left[ Z_{j}|\mathcal{F}_{j-1}\right] =\exp \left( -\left(
	e-2\right) \left( \frac{\beta }{n}\right) ^{2}\mathbb{E}\left[ R_{j}|%
	\mathcal{F}_{j-1}\right] \right) \mathbb{E}\left[ e^{\beta Y_{j}}|\mathcal{F}%
	_{j-1}\right] \leq Z_{j-1}.
	\end{equation*}%
	It follows that $\mathbb{E}\left[ Z_{n}\right] \leq 1$. Spelled out this is%
	\begin{equation*}
	1\geq \mathbb{E}\left[ \exp \left( \beta \left( F-V\right) -\frac{\left(
		e-2\right) \beta ^{2}}{n}F\right) \right] .
	\end{equation*}%
	If we choose $\beta =n/\left( 2\left( e-2\right) \right) <n$, then%
	\begin{equation*}
	1\geq \mathbb{E}\left[ \exp \left( \left( \frac{n}{4\left( e-2\right) }%
	\right) \left( F-2V\right) \right) \right] .
	\end{equation*}
\end{proof}

The proof of part (iii) needs one more lemma.\bigskip 

\begin{lemma}
	\label{Lemma Ball}For $\left( X_{1},...,X_{m}\right) \sim \mu ^{m}$ and $%
	k\in \left[ m\right] $ we have $\mathbb{E}\left[ \mu \left( B\left(
	X_{k},r\right) \backslash \bigcup_{i\in \left[ m\right] ,i\neq k}B\left(
	X_{i},r\right) \right) \right] \leq 1/m$.\bigskip
\end{lemma}

\begin{proof}
	For $X$ iid to $X_{i}$ the events $\left\{ X\in B\left( X_{k},r\right)
	\backslash \bigcup_{i\in \left[ m\right] ,i\neq k}B\left( X_{i},r\right)
	\right\} $ are disjoint for different values of $k$. It follows that their
	probabilities sum to at most $1$, and since by symmetry they have to be
	equal, the conclusion follows.\bigskip
\end{proof}

\begin{proof}[Proof of Theorem \protect\ref{Theorem Martingale estimator}]
	(i) Let $X$ be iid to the $X_{i}$ and for $k\in \left\{ n-m+1,n\right\} $
	let $R_{k}=\mathbf{1}\left\{ X_{k}\in \bigcap_{i<k}B\left( X_{i}\right)
	^{c}\right\} $, so $T_{m}\left( \mathbf{X}\right) =\left( 1/m\right)
	\sum_{k=n-m+1}^{n}R_{k}$ and $\mu \left( \bigcap_{i<k}B\left( X_{i}\right)
	^{c}\right) =\mathbb{E}\left[ R_{k}|\mathbf{X}_{1}^{k-1}\right] $.%
	\begin{eqnarray*}
		\hat{M}\left( \mathbf{X}\right)  &=&\frac{1}{m}\sum_{k=n-m+1}^{n}\mu \left(
		\bigcap_{i=1}^{n}B\left( X_{i}\right) ^{c}\right)  \\
		&\leq &\frac{1}{m}\sum_{k=n-m+1}^{n}\mu \left( \bigcap_{i<k}B\left(
		X_{i}\right) ^{c}\right) =\sum_{k=n-m+1}^{n}\frac{1}{m}\mathbb{E}\left[
		R_{k}|\mathbf{X}_{1}^{k-1}\right] .
	\end{eqnarray*}%
	Thus 
	\begin{equation*}
	\hat{M}\left( \mathbf{X}\right) -T_{m}\left( \mathbf{X}\right) \leq
	\sum_{k=n-m+1}^{n}\frac{1}{m}\left( \mathbb{E}\left[ R_{k}|\mathbf{X}%
	_{1}^{k-1}\right] -R_{k}\right) .
	\end{equation*}%
	Then $\left( 1/m\right) \left( \mathbb{E}\left[ R_{k}|\mathbf{X}_{1}^{k-1}%
	\right] -R_{k}\right) $ is a martingale difference sequence with values in $%
	\left[ -1/m,1/m\right] $. It follows from the Hoeffding-Azuma Theorem \cite%
	{McDiarmid98} that%
	\begin{equation*}
	\mathbb{P}\left\{ \hat{M}\left( \mathbf{X}\right) -T_{\mathbf{w}}\left( 
	\mathbf{X}\right) >t\right\} \leq e^{-mt^{2}/2}.
	\end{equation*}
	
	(ii) Use Lemma \ref{Lemma Martingale} with the same $R_{k}$, $F=\hat{M}%
	\left( \mathbf{X}\right) $, $V=T_{m}\left( \mathbf{X}\right) $ and $n$
	replaced by $m$ to obtain 
	\begin{equation*}
	1\geq \mathbb{E}\left[ \exp \left( \left( \frac{m}{4\left( e-2\right) }%
	\right) \left( \hat{M}\left( \mathbf{X}\right) -2T_{m}\left( \mathbf{X}%
	\right) \right) \right) \right] .
	\end{equation*}%
	Then (ii) follows from Markov's inequality.
	
	(iii) Observe that 
	\begin{equation*}
	\mathbb{E}\left[ T_{m}\left( \mathbf{X}\right) \right] \leq \frac{1}{m}%
	\sum_{k=n-m+1}^{n}\mathbb{E}\left[ 1\left\{ X_{k}\in
	\bigcap_{i=1}^{n-m}B\left( X_{i},r\right) ^{c}\right\} \right] =\mathbb{E}%
	\left[ \hat{M}\left( \mathbf{X}_{1}^{n-m}\right) \right] .
	\end{equation*}%
	On the other hand, using Lemma \ref{Lemma Ball} and $\ln t\leq 1-t$, 
	\begin{eqnarray*}
		\mathbb{E}\left[ \hat{M}\left( \mathbf{X}_{1}^{n-m}\right) -\hat{M}\left( 
		\mathbf{X}_{1}^{n}\right) \right]  &=&\sum_{k=n-m+1}^{n}\mathbb{E}\left[ \mu
		\left( B_{k}\backslash \bigcup_{j:j\leq k}B\left( X_{j},r\right) \right) %
		\right]  \\
		&\leq &\sum_{k=n-m+1}^{n}\frac{1}{k}\leq \int_{n-m}^{n}\frac{dt}{t}=\ln 
		\frac{n}{n-m} \\
		&\leq &\frac{m}{n-m}.
	\end{eqnarray*}%
	Thus $\mathbb{E}\left[ T_{m}\left( \mathbf{X}\right) -\hat{M}\left( \mathbf{X%
	}_{1}^{n}\right) \right] \leq \mathbb{E}\left[ \hat{M}\left( \mathbf{X}%
	_{1}^{n-m}\right) -\hat{M}\left( \mathbf{X}_{1}^{n}\right) \right] \leq \ln
	\left( n/\left( n-m\right) \right) \leq m/\left( n-m\right) $.\bigskip 
\end{proof}

\bigskip 

\begin{corollary}
	\label{Corollary bound missing mass with net Restatement}(\textbf{\ =
		Corollary \ref{Corollary bound missing mass with net}}) For $m\in \left[ n%
	\right] $ and $\delta >0$, with probability at least $1-\delta $ in $\mathbf{%
		X}$, if there is an $r$-net $\mathbf{Y\subset X}$, with $\left\vert \mathbf{Y%
	}\right\vert =m$ then 
	\begin{equation*}
	\hat{M}\left( \mathbf{X},r\right) \leq \frac{m}{n}+\sqrt{\frac{m\ln \left(
			n/\delta \right) }{n}}
	\end{equation*}%
	\bigskip 
\end{corollary}

\begin{proof}
	For any fixed $r$-net $\mathbf{Y\subset X}$ of cardinality $m$ reorder $%
	\mathbf{X}$ by putting $\mathbf{Y}$ to the front of the sequence. Then%
	\begin{eqnarray*}
		T\left( \mathbf{X},r\right)  &\leq &\frac{1}{n}\sum_{k=1}^{m}1\left\{
		Y_{k}\in \bigcap_{i<k}B\left( Y_{k},r\right) ^{c}\right\} +\frac{1}{n}%
		\sum_{k=m+1}^{n}1\left\{ X_{k}\notin \bigcup_{i=1}^{m}B\left( Y_{k},r\right)
		\right\}  \\
		&\leq &\frac{m}{n},
	\end{eqnarray*}%
	because the $Y_{k}$ are an $r$-net. The result follows from combining
	Theorem \ref{Theorem Martingale estimator} with a union bound over all
	subsets of cardinality $m$.\bigskip 
\end{proof}

\subsection{A negative result\label{Section Proof A negative result}}

\begin{proposition}[\textbf{= Proposition \protect\ref{Proposition Negative}}%
	]
	\label{Proposition Negative Restatement}Let $1<r<\sqrt{2}$. For every $%
	\epsilon \in \left( 0,1\right) $ and $n\in \mathbb{N}$ with $n\geq \ln
	\left( 4\right) /\epsilon $ there exists $D\in \mathbb{N}$ and $\mu $ on $%
	\mathbb{R}^{D}$ such that
	
	(i) for $\mathbf{X}\sim \mu ^{n}$, $\min \left\{ \mathbb{V}\left( \hat{M}%
	\left( \mathbf{X},r\right) \right) ,\mathbb{V}\left( G\left( \mathbf{X}%
	,r\right) \right) \right\} \geq \left( 1/4\right) -\epsilon $.
	
	(ii) Let $B$ be the event $\left\{ \forall i,j\text{ with }i\neq
	j,\left\Vert X_{i}-X_{j}\right\Vert >r\text{ and }\left\Vert
	X_{i}\right\Vert \leq 1\right\} $. Then for $D$ sufficiently large $\mathbb{P%
	}\left( B\right) \geq 1/2-\epsilon $.
	
	(iii) For every $f:\mathcal{X}^{n}\rightarrow \mathbb{R}$ there exists $\mu
	^{\prime \prime }$ on $\mathbb{R}^{D}$ such that for $\mathbf{X}\sim \left(
	\mu ^{\prime \prime }\right) ^{n}$, we have 
	\begin{equation*}
	\mathbb{E}\left[ \left( f\left( \mathbf{X}\right) -M\left( \mu ^{\prime
		\prime },n,r\right) \right) ^{2}\right] \geq \left( 1-\epsilon \right)
	^{2}/16,
	\end{equation*}%
	and consequently $\left\Vert M-f\left( \mathbf{X}\right) \right\Vert
	_{L_{2}\left( \mu ^{n}\right) }\geq \left( 1-\epsilon \right) /4$.
\end{proposition}

\begin{proof}
	Let $D\geq 2n/\epsilon $ and choose $\ r$ with $1<r<\sqrt{2}$.
	
	Now let $\mu =\left( 1/2\right) ^{1/n}\left( 1/D\right) \sum_{i=1}^{D}\delta
	_{e_{i}}+\left( 1-\left( 1/2\right) ^{1/n}\right) \delta _{0}$ and let $%
	\mathbf{X}$ be an $n$-sample drawn from $\mu $. Let $A$ be the event that $0$
	occurs in $\mathbf{X}$. Then $\mathbb{P}A=1/2$ by definition of $\mu $,
	since $\mathbb{P}A^{c}=\left( 1-\left( 1-\left( 1/2\right) ^{1/n}\right)
	\right) ^{n}=1/2$. If $A$ occurs then $\hat{M}\left( \mathbf{X},r\right) =0$%
	, because all basis vectors are within $r$ from $0$. Under $A^{c}$ however
	the sample must miss $D-n$ basis vectors, so $\hat{M}\left( \mathbf{X}%
	,r\right) \geq \left( 1/2\right) ^{1/n}\left( 1-n/D\right) $. Thus $\left(
	1/2\right) -2\epsilon \leq \left( 1/2\right) ^{1/n}\left( 1-n/D\right)
	/2\leq M\left( \mu ,n,r\right) \leq 1/2$ and 
	\begin{eqnarray*}
		Var\left( \hat{M}\left( \mathbf{X},r\right) \right) &\geq &\left( 1/2\right)
		\left( \left( 1/4\right) ^{1/n}\left( 1-n/D\right) ^{2}\right) -1/4 \\
		&\geq &\left( 1/2\right) \left( 1-\left( 1/n\right) \ln 4\right) \left(
		1-2n/D\right) -1/4 \\
		&\geq &\left( 1/2\right) \left( 1-\epsilon \right) ^{2}-1/4\geq \left(
		1/4\right) -\epsilon .
	\end{eqnarray*}
	
	Before we come to the Good-Turing estimator we prove (ii). Let $B$ be the
	event in (ii)\ which just means that $\mathbf{X}$ consists of $n$ distinct
	basis vectors. Similar to the reasoning in the birthday paradox the
	probability of $B$ is 
	\begin{eqnarray*}
		\mathbb{P}_{\mu ^{n}}\left( B\right) &=&\frac{1}{2}\prod_{i=1}^{n}\left( 1-%
		\frac{i-1}{D}\right) \geq \frac{1}{2}\left( 1-\frac{n-1}{D}\right) ^{n}=%
		\frac{1}{2}\exp \left( n\ln \left( 1-\frac{n-1}{D}\right) \right) \\
		&\geq &\frac{1}{2}\left( 1-\frac{n^{2}}{D-n}\right) \geq \frac{1}{2}%
		-\epsilon \text{,}
	\end{eqnarray*}%
	by making $D$ sufficiently large, which gives (ii). Under $B$ we have $G=1$.
	But under $A$ we have $G=0$ with probability $1/2$. It follows that $\mathbb{%
		E}\left[ G\right] \leq 1/2$ and 
	\begin{equation*}
	Var\left[ G\right] \geq \left( 1/2\right) 0^{2}+\left( 1/2-\epsilon \right)
	1^{2}-1/4=1/4-\epsilon \text{,}
	\end{equation*}%
	which completes the proof of (i).
	
	(iii) Now define $\mu ^{\prime }=\left( 1/D\right) \sum_{i=1}^{D}\delta
	_{e_{i}}$ and let $\mathbf{Y}\sim \left( \mu ^{\prime }\right) ^{n}$. Then $%
	M\left( \mu ^{\prime },n,r\right) \geq 1-n/D\geq 1-\epsilon /2$ and $M\left(
	\mu ^{\prime },n,r\right) -M\left( \mu ,n,r\right) \geq \left( 1-\epsilon
	\right) /2$. But conditional on $A^{c}$ the samples $\mathbf{X}$ and $%
	\mathbf{Y}$ are identically distributed, so%
	\begin{align*}
	& \mathbb{E}\left[ \left( f\left( \mathbf{Y}\right) -M\left( \mu ^{\prime
	},n,r\right) \right) ^{2}+\left( f\left( \mathbf{X}\right) -M\left( \mu
	,n,r\right) \right) ^{2}\right] \\
	& \geq \mathbb{E}\left[ \left( f\left( \mathbf{Y}\right) -M\left( \mu
	^{\prime },n,r\right) \right) ^{2}+\left( f\left( \mathbf{Y}\right) -M\left(
	\mu ,n,r\right) \right) ^{2}|A^{c}\right] \Pr \left( A^{c}\right) \\
	& \geq \frac{\left( M\left( \mu ^{\prime },n,r\right) -M\left( \mu
		,n,r\right) \right) ^{2}}{2}\geq \frac{\left( 1-\epsilon \right) ^{2}}{8},
	\end{align*}%
	which gives (ii) with either with $\mu ^{\prime \prime }=\mu $ or $\mu
	^{\prime \prime }=\mu ^{\prime }$. In the second inequality we used calculus
	to minimize $\left( x-M\left( \mu _{1},n,r\right) \right) ^{2}+\left(
	x-M\left( \mu _{2},n,r\right) \right) ^{2}$.\bigskip
\end{proof}

\subsection{Local separation\label{Section Proof Condensation and separation}%
}

\begin{theorem}[\textbf{= Theorem \protect\ref{Theorem Main}}]
	\label{Theorem Main Restatement} Under the conventions of Section \ref%
	{Section Notation and conventions} let $n\geq 16$. Then%
	\begin{eqnarray*}
		\mathbb{V}\left[ G\left( \mathbf{X},r\right) \right]  &\leq &\frac{2\left( 1+%
			\mathbb{E}\left[ h\left( \mathbf{X},r\right) \right] \right) }{n} \\
		\mathbb{V}\left[ \hat{M}\left( \mathbf{X},r\right) \right]  &\leq &\frac{2%
			\mathbb{E}\left[ h\left( \mathbf{X},r\right) \right] +4\left( e-2\right)
			\left( \ln n+1\right) }{n-1}.
	\end{eqnarray*}%
	Furthermore, for any $t>0$,%
	\begin{eqnarray*}
		\mathbb{P}\left\{ \left\vert G\left( \mathbf{X},r\right) -\mathbb{E}\left[
		G\left( \mathbf{X},r\right) \right] \right\vert >12\sqrt{\frac{\left( 1+%
				\mathbb{E}\left[ h\left( \mathbf{X},r\right) \right] \right) t}{n}}+\frac{23t%
		}{\sqrt{n}}\right\}  &\leq &15e^{-t} \\
		\mathbb{P}\left\{ \left\vert \hat{M}\left( \mathbf{X},r\right) -\mathbb{E}%
		\left[ \hat{M}\left( \mathbf{X},r\right) \right] \right\vert >12\sqrt{\frac{%
				\mathbb{E}\left[ h\left( \mathbf{X},r\right) \right] t}{n}}+\frac{37t}{\sqrt{%
				n-1}}\right\}  &\leq &2ne^{-t}.
	\end{eqnarray*}%
	\bigskip 
\end{theorem}

Define a nonlinear operator $Q$ acting on bounded functions $f:\mathcal{X}%
^{n}\rightarrow \mathbb{R}$ by%
\begin{equation*}
Qf\left( \mathbf{x}\right) =f\left( \mathbf{x}\right) -\min_{k}\inf_{y\in 
	\mathcal{X}}~f\left( S_{y}^{k}\left( \mathbf{x}\right) \right)
=\max_{k}\sup_{y\in \mathcal{X}}f\left( \mathbf{x}\right) -f\left(
S_{y}^{k}\left( \mathbf{x}\right) \right) .
\end{equation*}%
The proof of Theorem \ref{Theorem Main} uses the following general
concentration inequality, which may be of independent interest. Its proof is
given in the next section.

\begin{proposition}
	\label{Mytheorem}Let $\mathbf{X}=\left( X_{1},...,X_{n}\right) $ be a vector
	of independent random variables with values in $\mathcal{X}$ and \thinspace $%
	f:\mathcal{X}^{n}\rightarrow \left[ 0,1\right] $ be measurable and strongly $%
	\left( a,0\right) $-self-bounded in the sense that%
	\begin{equation*}
	\forall \mathbf{x}\in \mathcal{X}^{n}\text{, }\sum_{k=1}^{n}f\left( \mathbf{x%
	}\right) -\inf_{y\in \mathcal{X}}f\left( S_{y}^{k}\left( \mathbf{x}\right)
	\right) \leq af\left( \mathbf{x}\right)
	\end{equation*}%
	with $a\geq 1$. Then $\mathbb{V}\left[ f\left( \mathbf{X}\right) \right]
	\leq a\mathbb{E}\left[ Qf\left( \mathbf{X}\right) \right] $. Suppose also
	that for some $b\geq 1$ and $w,\lambda >0$ and for all $t>0$%
	\begin{equation*}
	\mathbb{P}\left\{ Qf\left( \mathbf{X}\right) >w+t\right\} \leq be^{-\lambda
		t}\text{.}
	\end{equation*}%
	Then with $C\approx 4.16$ we have for every $\delta \in \left( 0,1\right) $%
	\begin{equation*}
	\mathbb{P}\left\{ \left\vert f\left( \mathbf{X}\right) -\mathbb{E}\left[
	f\left( \mathbf{X}\right) \right] \right\vert >\sqrt{Cae^{2}w\ln \left(
		b+2e^{2}/\delta \right) }+e^{2}\sqrt{\frac{Ca}{\lambda }}\ln \left(
	b+2e^{2}/\delta \right) \right\} \leq \delta .
	\end{equation*}%
	and for $t>0$%
	\begin{equation*}
	\mathbb{P}\left\{ \left\vert f\left( \mathbf{X}\right) -\mathbb{E}\left[
	f\left( \mathbf{X}\right) \right] \right\vert >t\right\} \leq 2\left(
	b+e^{2}\right) \exp \left( \frac{-t^{2}}{e^{2}\left( Caw+2\sqrt{Ca\lambda
			^{-1}}t\right) }\right) .
	\end{equation*}%
	If $b=1$ then $b$ can be deleted from these inequalities.\bigskip
\end{proposition}

To apply this proposition we will show that $\hat{M}$ and $G$ satisfy the
above hypotheses. Define for $k\in \left\{ 1,...,n\right\} $ functions $%
W_{k} $ and $W:\mathcal{X}^{n}\rightarrow \mathbb{R}$%
\begin{equation*}
W_{k}\left( \mathbf{x}\right) :=\Pr \left\{ B\left( x_{k}\right) \backslash
\bigcup_{i:i\neq k}B\left( x_{i}\right) \right\} \text{ and }W\left( \mathbf{%
	x}\right) :=\max_{k}W_{k}\left( \mathbf{x}\right) .
\end{equation*}

\begin{lemma}
	\label{Lemma Bound DL}$\hat{M}^{\perp }$ is $\left( 1,0\right) $%
	-self-bounded and $Q\hat{M}^{\perp }\leq W$.
\end{lemma}

\begin{proof}
	With reference to any $k\in \left\{ 1,...,n\right\} $%
	\begin{equation*}
	\hat{M}^{\perp }\left( \mathbf{x}\right) =\mu \left( \bigcup_{i}B\left(
	x_{i}\right) \right) =\mu \left( \bigcup_{i:i\neq k}B\left( x_{i}\right)
	\right) +W_{k}\left( \mathbf{x}\right) .
	\end{equation*}%
	It follows that $\hat{M}^{\perp }\left( \mathbf{x}\right) -\inf_{y}\hat{M}%
	^{\perp }\left( S_{y}^{k}\mathbf{x}\right) \leq W_{k}\left( \mathbf{x}%
	\right) $ and thus $Q\hat{M}^{\perp }\leq W$. Also note that%
	\begin{equation*}
	\sum_{k}W_{k}\left( \mathbf{x}\right) =\sum_{k}\mu \left( B\left(
	x_{k}\right) \backslash \bigcup_{i\neq k}B\left( x_{i}\right) \right) =\mu
	\left( \bigcup_{k}\left( B\left( x_{k}\right) \backslash \bigcup_{i\neq
		k}B\left( x_{i}\right) \right) \right) \leq \hat{M}^{\perp }\left( \mathbf{x}%
	\right) ,
	\end{equation*}%
	since the events in the second sum are disjoint.\bigskip
\end{proof}

\begin{lemma}
	\label{Lemma Bound DR}$G^{\perp }$ is $\left( 2,0\right) $-self-bounded and $%
	QG^{\perp }\leq \left( 1+h\right) /n$.
\end{lemma}

\begin{proof}
	With reference to any $k\in \left\{ 1,...,n\right\} $, with a disjoint
	decomposition as in the proof of Lemma \ref{Lemma Bound DL},%
	\begin{eqnarray*}
		G^{\perp }\left( \mathbf{x}\right) &=&\frac{1}{n}\sum_{j=1}^{n}\mathbf{1}%
		\left\{ x_{j}\in U_{j}\right\} \\
		&=&\frac{1}{n}\mathbf{1}\left\{ x_{k}\in U_{k}\right\} +\frac{1}{n}%
		\sum_{j:j\neq k}\mathbf{1}\left\{ x_{j}\in U_{jk}\cup \left( U_{j}\backslash
		U_{jk}\right) \right\} \\
		&=&\frac{1}{n}\mathbf{1}\left\{ x_{k}\in U_{k}\right\} +\frac{1}{n}%
		\sum_{j:j\neq k}\mathbf{1}\left\{ x_{j}\in U_{jk}\right\} +\frac{1}{n}%
		\sum_{j:j\neq k}\mathbf{1}\left\{ x_{j}\in B_{k}\backslash U_{jk}\right\} .
	\end{eqnarray*}%
	The middle term is independent of $x_{k}$ and the subsequence of points $%
	x_{j}$, which contribute to the sum in the last term, has the local
	separation property, so this term is bounded by $h\left( \mathbf{x}\right)
	/n $. It follows that 
	\begin{eqnarray*}
		G^{\perp }\left( \mathbf{x}\right) -\inf_{y}G^{\perp }\left( S_{y}^{k}%
		\mathbf{x}\right) &\leq &\frac{1}{n}\mathbf{1}\left\{ x_{k}\in U_{k}\right\}
		+\frac{1}{n}\sum_{j:j\neq k}\mathbf{1}\left\{ x_{j}\in B_{k}\backslash
		U_{jk}\right\} \\
		&\leq &\left( 1+h\left( \mathbf{x}\right) \right) /n
	\end{eqnarray*}%
	and likewise $QG^{\perp }\left( \mathbf{x}\right) \leq \left( 1+h\left( 
	\mathbf{x}\right) \right) /n$. Also from the above%
	\begin{align*}
	& \sum_{k}G^{\perp }\left( \mathbf{x}\right) -\inf_{y}G^{\perp }\left(
	S_{y}^{k}\mathbf{x}\right) \\
	& \leq \frac{1}{n}\sum_{k}\mathbf{1}\left\{ x_{k}\in U_{k}\right\} +\frac{1}{%
		n}\sum_{k}\sum_{j:j\neq k}\mathbf{1}\left\{ x_{j}\in B_{k}\backslash
	U_{jk}\right\} \\
	& =G^{\perp }\left( \mathbf{x}\right) +\frac{1}{n}\sum_{j}\sum_{k:k\neq j}%
	\mathbf{1}\left\{ x_{j}\in B_{k}\backslash U_{jk}\right\} \text{ \ \ (*)} \\
	& =G^{\perp }\left( \mathbf{x}\right) +\frac{1}{n}\sum_{j}\mathbf{1}\left\{
	x_{j}\in \bigcup_{k:k\neq j}\left( B_{k}\backslash U_{jk}\right) \right\} \\
	& \leq 2G^{\perp }\left( \mathbf{x}\right) ,
	\end{align*}%
	since the sets in the sum over $k$ in (*) are disjoint.\bigskip
\end{proof}

\begin{lemma}
	\label{Lemma Bound Wk against h}For $t>0$ and $k\in \left\{ 1,...,n\right\} $%
	\begin{eqnarray*}
		\mathbb{P}\left\{ W_{k}\left( \mathbf{X}\right) -\frac{2h\left( \mathbf{X}%
			\right) }{n-1}>t\right\} &\leq &\exp \left( \frac{-\left( n-1\right) t}{%
			4\left( e-2\right) }\right) \text{ and} \\
		\mathbb{P}\left\{ W\left( \mathbf{X}\right) -\frac{2h\left( \mathbf{X}%
			\right) }{n-1}>t\right\} &\leq &n\exp \left( \frac{-\left( n-1\right) t}{%
			4\left( e-2\right) }\right) .
	\end{eqnarray*}
\end{lemma}

\begin{proof}
	For $k,j\in \left\{ 1,...,n\right\} $, $k\neq j$ let $R_{j}^{k}$ be the
	random variable 
	\begin{equation*}
	R_{j}^{k}=\mathbf{1}\left\{ X_{j}\in B_{k}\backslash \bigcup_{i:i\neq
		k,i<j}B_{i}\right\}
	\end{equation*}%
	$R_{j}^{k}$ has values in $\left[ 0,1\right] $, and $R_{j}^{k}$ is $\mathcal{%
		F}_{j}$-measurable, where $\mathcal{F}_{j}=\Sigma \left( X_{k},X_{i}\right)
	_{i\leq j}$. Then 
	\begin{eqnarray*}
		W_{k}\left( \mathbf{X}\right) &=&\frac{1}{n-1}\sum_{j:j\neq k}\mu \left\{
		B_{k}\backslash \bigcup_{i:i\neq k}B_{i}\right\} \\
		&\leq &\frac{1}{n-1}\sum_{j:j\neq k}\mu \left\{ B_{k}\backslash
		\bigcup_{i:i\neq k,i<j}B_{i}\right\} \\
		&=&\frac{1}{n-1}\sum_{j:j\neq k}\mathbb{E}\left[
		R_{j}^{k}|X_{k},X_{1},...,X_{j-1}\right] =F_{k}\left( \mathbf{X}\right) .
	\end{eqnarray*}%
	Let 
	\begin{equation*}
	V_{k}\left( \mathbf{X}\right) =\frac{1}{n-1}\sum_{j:j\neq k}R_{j}^{k}=\frac{1%
	}{n-1}\sum_{j:j\neq k}\mathbf{1}\left\{ X_{j}\in B_{k}\backslash
	\bigcup_{i:i\neq k,i<j}B_{i}\right\} .
	\end{equation*}%
	Note that the indices $j$ which contribute to the sum in $V_{k}\left( 
	\mathbf{x}\right) $ must be such that each $X_{j}$ is in the ball about $%
	X_{k}$, but none of them may be in the ball about any other one of the
	contributing indices. It follows that the corresponding subsequence has the
	local separation property. Therefore $V_{k}\left( \mathbf{X}\right) \leq
	h\left( \mathbf{X}\right) /\left( n-1\right) $.
	
	Lemma \ref{Lemma Martingale} applied conditional on $X_{k}$ gives us 
	\begin{equation*}
	1\geq \mathbb{E}\left[ \exp \left( \left( \frac{n-1}{4\left( e-2\right) }%
	\right) \left( F_{k}\left( \mathbf{X}\right) -2V_{k}\left( \mathbf{X}\right)
	\right) \right) |X_{k}\right] .
	\end{equation*}%
	Of course the unconditional expectation of the R.H.S. will also be bounded
	by $1$. Markov's inequality then implies%
	\begin{eqnarray*}
		\mathbb{P}\left\{ W_{k}\left( \mathbf{X}\right) -\frac{2h\left( \mathbf{X}%
			\right) }{n-1}>t\right\} &\leq &\mathbb{P}\left\{ F_{k}\left( \mathbf{X}%
		\right) >2V_{k}\left( \mathbf{X}\right) +t\right\} \\
		&\leq &\exp \left( \frac{-\left( n-1\right) t}{4\left( e-2\right) }\right) .
	\end{eqnarray*}%
	The second statement follows from a union bound.\bigskip
\end{proof}

\begin{corollary}
	\label{Lemma bound EH by sample Restatement} (\textbf{= Corollary \ref{Lemma
			bound EH by sample}) }For $t>0$
	
	(i) $\mathbb{P}\left\{ \sqrt{\mathbb{E}\left[ h\left( \mathbf{X},r\right) %
		\right] }\leq \sqrt{h\left( \mathbf{X},r\right) }+\sqrt{2t}\right\} \geq
	1-e^{-t}$
	
	(ii) $\mathbb{P}\left\{ h\left( \mathbf{X},r\right) -2\mathbb{E}\left[
	h\left( \mathbf{X},r\right) \right] >t\right\} \leq e^{-6t/7}.$
\end{corollary}

\begin{proof}
	Equating the r.h.s. of Theorem \ref{Theorem SelfBound} (ii) to $\delta $ and
	solving for $t$ gives for $\delta >0$ with probability at least $1-\delta $
	that $\mathbb{E}\left[ h\left( \mathbf{X}\right) \right] -h\left( \mathbf{X}%
	\right) \leq \sqrt{2\mathbb{E}\left[ h\left( \mathbf{X}\right) \right] \ln
		\left( 1/\delta \right) }$. Bringing the r.h.s. to the left, completing the
	square and taking the square root gives (i) with $\delta =e^{-t}$. Similarly
	we get from Theorem \ref{Theorem SelfBound} (i) with probability at least $%
	1-\delta $ that%
	\begin{equation*}
	h\left( \mathbf{X}\right) -E\left[ h\left( \mathbf{X}\right) \right] \leq 
	\sqrt{2\mathbb{E}\left[ h\left( \mathbf{X}\right) \right] \ln \left(
		1/\delta \right) }+\frac{2\ln \left( 1/\delta \right) }{3}.
	\end{equation*}%
	Then use $\sqrt{2\mathbb{E}\left[ h\left( \mathbf{X}\right) \right] \ln
		\left( 1/\delta \right) }\leq \mathbb{E}\left[ h\left( \mathbf{X}\right) %
	\right] +\ln \left( 1/\delta \right) /2$ and set $\delta =e^{-t}$ to get the
	second conclusion.\bigskip
\end{proof}

\begin{proof}[Proof of Theorem \protect\ref{Theorem Main}]
	Lemma \ref{Lemma Bound Wk against h} and integration by parts gives\ for $%
	\delta >0$%
	\begin{eqnarray*}
		\mathbb{E}\left[ W\left( \mathbf{X}\right) -\frac{2h\left( \mathbf{X}\right) 
		}{n-1}\right]  &=&\delta +\int_{\delta }^{\infty }\mathbb{P}\left\{
		\max_{k}W_{k}\left( \mathbf{X}\right) -\frac{2h\left( \mathbf{X}\right) }{n-1%
		}>t\right\} dt \\
		&\leq &\delta +n\int_{\delta }^{\infty }\exp \left( \frac{-\left( n-1\right)
			t}{4\left( e-2\right) }\right) dt \\
		&=&\delta +\frac{4n\left( e-2\right) }{n-1}\exp \left( \frac{-\left(
			n-1\right) \delta }{4\left( e-2\right) }\right) .
	\end{eqnarray*}%
	With $\delta =4n\left( e-2\right) \ln \left( n\right) /\left( n-1\right) $
	we obtain 
	\begin{equation*}
	\mathbb{E}\left[ Q\hat{M}^{\perp }\left( \mathbf{X}\right) \right] \leq 
	\mathbb{E}\left[ W\left( \mathbf{X}\right) \right] \leq \frac{2\mathbb{E}%
		\left[ h\left( \mathbf{X}\right) \right] }{n-1}+\frac{4\left( e-2\right)
		\left( \ln n+1\right) }{n-1},
	\end{equation*}%
	so Proposition \ref{Mytheorem} gives us the bound on the variance of $\hat{M}%
	\left( \mathbf{X}\right) $. The variance bound for $G$ follows from
	Proposition \ref{Mytheorem} and Lemma \ref{Lemma Bound DR}.
	
	From Corollary \ref{Lemma bound EH by sample} we get for $t>0$ 
	\begin{equation}
	\mathbb{P}\left\{ \frac{1+h\left( \mathbf{X}\right) }{n}>\frac{1+2\mathbb{E}%
		\left[ h\left( \mathbf{X}\right) \right] }{n}+t\right\} \leq e^{-\left(
		6/7\right) nt}.  \label{Tailbound DG}
	\end{equation}%
	Combined with Lemma \ref{Lemma Bound Wk against h} we obtain%
	\begin{eqnarray}
	\mathbb{P}\left\{ W\left( \mathbf{X}\right) -\frac{4\mathbb{E}\left[ h\left( 
		\mathbf{X}\right) \right] }{n}>t\right\}  &\leq &n\exp \left( \frac{-\left(
		n-1\right) t}{8\left( e-2\right) }\right) +e^{-\left( 6/14\right) nt}  \notag
	\\
	&\leq &\left( n+1\right) \exp \left( \frac{-\left( n-1\right) t}{8\left(
		e-2\right) }\right)   \label{Tailbound W}
	\end{eqnarray}%
	We summarize:
	
	Lemma \ref{Lemma Bound DR} and (\ref{Tailbound DG}) imply that we can use
	Proposition \ref{Mytheorem} with $f=G^{\perp }$ and the values $a=2$, $b=1$, 
	$\lambda =\left( 6/7\right) n$ and $w=\left( 1+2\mathbb{E}\left[ h\left( 
	\mathbf{X}\right) \right] \right) /n$. Substitution gives%
	\begin{equation*}
	\mathbb{P}\left\{ \left\vert G^{\perp }\left( \mathbf{X}\right) -\mathbb{E}%
	\left[ G^{\perp }\left( \mathbf{X}\right) \right] \right\vert >\sqrt{\frac{%
			2Ce^{2}\left( 1+2\mathbb{E}\left[ h\left( \mathbf{X}\right) \right] \right)
			\ln \left( 2e^{2}/\delta \right) }{n}}+e^{2}\sqrt{\frac{14C}{6n}}\ln \left(
	2e^{2}/\delta \right) \right\} \leq \delta .
	\end{equation*}%
	With some simplifications we get for $t>0$%
	\begin{equation*}
	\mathbb{P}\left\{ \left\vert G\left( \mathbf{X}\right) -\mathbb{E}\left[
	G\left( \mathbf{X}\right) \right] \right\vert >12\sqrt{\frac{\left( 1+%
			\mathbb{E}\left[ h\left( \mathbf{X}\right) \right] \right) t}{n}}+\frac{23t}{%
		\sqrt{n}}\right\} \leq 15e^{-t}.
	\end{equation*}
	
	Lemma \ref{Lemma Bound DL} and (\ref{Tailbound W}) imply that we can use
	Proposition \ref{Mytheorem} with $f=\hat{M}^{\perp }$ and the values $a=1$, $%
	b=n+1$, $\lambda =\left( n-1\right) /\left( 8\left( e-2\right) \right) $ and 
	$w=4\mathbb{E}\left[ h\left( \mathbf{X}\right) \right] /n$. Substitution
	gives%
	\begin{equation*}
	\mathbb{P}\left\{ \left\vert f\left( \mathbf{X}\right) -\mathbb{E}\left[
	f\left( \mathbf{X}\right) \right] \right\vert >\sqrt{\frac{Ce^{2}4\mathbb{E}%
			\left[ h\left( \mathbf{X}\right) \right] \ln \left( n+1+2e^{2}/\delta
			\right) }{n}}+e^{2}\sqrt{\frac{8C\left( e-2\right) }{n-1}}\ln \left(
	n+1+2e^{2}/\delta \right) \right\} \leq \delta .
	\end{equation*}%
	Using $n\geq 16>1+2e^{2}$ we can simplify and resolve the constants to
	obtain for $t>0$%
	\begin{equation*}
	\mathbb{P}\left\{ \left\vert \hat{M}\left( \mathbf{X}\right) -\mathbb{E}%
	\left[ \hat{M}\left( \mathbf{X}\right) \right] \right\vert >12\sqrt{\frac{%
			\mathbb{E}\left[ h\left( \mathbf{X}\right) \right] t}{n}}+\frac{37t}{\sqrt{%
			n-1}}\right\} \leq 2ne^{-t}.
	\end{equation*}
\end{proof}

\subsection{Proof of Proposition \protect\ref{Mytheorem}}

The proof uses the following moment inequalities first given in (\cite%
{boucheron2005moment}).

\begin{theorem}
	\label{Theorem Moment inequalities}(Theorems 15.5 and 15.7 in \cite%
	{Boucheron13}) Let $\mathbf{X}=\left( X_{1},...,X_{n}\right) $ be a vector
	of independent random variables with values in $\mathcal{X}$ and $f:\mathcal{%
		X}^{n}\rightarrow \mathbb{R}$. For $q\geq 2$ with $\kappa \approx 1.271$%
	\begin{equation*}
	\left\Vert \left( f\left( \mathbf{X}\right) -\mathbb{E}\left[ f\left( 
	\mathbf{X}\right) \right] \right) _{+}\right\Vert _{q}\leq \sqrt{\kappa
		q\left\Vert V^{+}f\left( \mathbf{X}\right) \right\Vert _{q/2}}
	\end{equation*}%
	and with $C\approx 4.16$%
	\begin{equation*}
	\left\Vert \left( f\left( \mathbf{X}\right) -\mathbb{E}\left[ f\left( 
	\mathbf{X}\right) \right] \right) _{-}\right\Vert _{q}\leq \sqrt{Cq\left(
		\left\Vert V^{+}f\left( \mathbf{X}\right) \right\Vert _{q/2}\vee q\left\Vert
		Qf\left( \mathbf{X}\right) \right\Vert _{q}^{2}\right) },
	\end{equation*}%
	where 
	\begin{equation*}
	V^{+}f\left( \mathbf{x}\right) =\sum_{k=1}^{n}\mathbb{E}_{X}\left[ \left(
	f\left( \mathbf{x}\right) -f\left( S_{X}^{k}\left( \mathbf{x}\right) \right)
	\right) _{+}^{2}\right] .
	\end{equation*}
\end{theorem}

We also need a few lemmata, one to convert exponential tail bounds to moment
bounds, and one to convert moment bounds to tail bounds.

\begin{lemma}
	\label{Lemma basic}Suppose that $X$, $w,\lambda $, $b\geq 0$, $p\geq 1$ and $%
	\forall t>0$%
	\begin{equation*}
	\mathbb{P}\left\{ X>w+t\right\} \leq be^{-\lambda t}.
	\end{equation*}%
	Then $\left\Vert X\right\Vert _{p}\leq 2\lambda ^{-1}b^{1/p}p+w$.
\end{lemma}

\begin{proof}
	We have $\left\vert X\right\vert =\left\vert X-w+w\right\vert \leq \left(
	X-w\right) _{+}+w$. Then for $p\geq 1$%
	\begin{eqnarray*}
		\mathbb{E}\left[ \left( \lambda \left( X-w\right) _{+}\right) ^{p}\right]
		&=&\int_{0}^{\infty }\mathbb{P}\left\{ \left( \lambda \left( X-w\right)
		_{+}\right) ^{p}>s\right\} ds \\
		&=&\int_{0}^{\infty }\mathbb{P}\left\{ \lambda \left( X-w\right)
		_{+}>t\right\} pt^{p-1}dt\text{ with }s=t^{p} \\
		&\leq &bp\int_{0}^{\infty }e^{-t}t^{p-1}dt=bp\Gamma \left( p\right) \leq
		bp\left( p\right) ^{p}\leq b\left( 2p\right) ^{p}.
	\end{eqnarray*}%
	So $\left\Vert \lambda \left( X-w\right) _{+}\right\Vert _{p}\leq 2b^{1/p}p$
	or $\left\Vert X\right\Vert _{p}\leq 2\lambda ^{-1}b^{1/p}p+w$.\bigskip
\end{proof}

\begin{lemma}
	\label{Lemma little}Suppose $c,d,t>0$ and $\sqrt{cx}+dx\geq t$. Then 
	\begin{equation*}
	x\geq \frac{t^{2}}{c+2dt}
	\end{equation*}
\end{lemma}

\begin{proof}
	If $t\leq dx$ then $x\geq t/d=t^{2}/\left( dt\right) \geq t^{2}/\left(
	c+2dt\right) $, so we can assume $t>dx$. Then $\sqrt{cx}+dx\geq t$ $\implies 
	\sqrt{cx+\left( dx\right) ^{2}}\geq t-dx$ $\implies cx+\left( dx\right)
	^{2}\geq \left( t-dx\right) ^{2}=t^{2}-2dxt+\left( dx\right) ^{2}\implies
	\left( c+2dt\right) x\geq t^{2}$.\bigskip
\end{proof}

\begin{lemma}
	\label{Lemma convert moment to confidence}Suppose for $\alpha ,\gamma
	>0,b\geq 1$ and $p\geq p_{\min }\geq 1$ we have $\left\Vert Y\right\Vert
	_{p}\leq \sqrt{\alpha p}+\gamma b^{1/p}p$. Then
	
	(i) for $\delta \in \left( 0,1\right) $%
	\begin{equation*}
	\mathbb{P}\left\{ \left\vert Y\right\vert >\sqrt{e^{2}\alpha \ln \left(
		b+e^{p_{\min }}/\delta \right) }+e^{2}\gamma \ln \left( b+e^{p_{\min
		}}/\delta \right) \right\} \leq \delta .
		\end{equation*}
		
		(ii) for $t>0$%
		\begin{equation*}
		\mathbb{P}\left\{ \left\vert Y\right\vert >t\right\} \leq \left(
		b+e^{p_{\min }}\right) \exp \left( \frac{-t^{2}}{e^{2}\left( \alpha +2\gamma
			t\right) }\right) .
		\end{equation*}
		
		(iii) If $b=1$ then $b$ can be deleted in both inequalities above.
	\end{lemma}
	
	\begin{proof}
		If $p\geq \max \left\{ p_{\min },\ln \left( 1/\delta \right) \right\} $ then%
		\begin{equation*}
		\mathbb{P}\left\{ \left\vert Y\right\vert >e\left( \sqrt{\alpha p}+\gamma
		b^{1/p}p\right) \right\} \leq \mathbb{P}\left\{ \left\vert Y\right\vert >e^{%
			\frac{\ln \left( 1/\delta \right) }{p}}\left\Vert Y\right\Vert _{p}\right\}
		\leq \left( \frac{\left\Vert Y\right\Vert _{p}}{\left\Vert Y\right\Vert
			_{p}e^{\frac{\ln \left( 1/\delta \right) }{p}}}\right) ^{p}=\delta .
		\end{equation*}%
		The first inequality follows from the assumed bound on $\left\Vert
		Y\right\Vert _{p}$, the second is Markov's. Setting $p=\ln \left( \left(
		b+e^{p_{\min }}\right) /\delta \right) $ we have $p\geq \max \left\{ p_{\min
		},\ln \left( 1/\delta \right) \right\} $ and also $p\geq \ln b$, so that $%
		b^{1/p}=e^{\left( \ln b\right) /p}\leq e$. Substitution gives (i).
		
		Let $\delta >0$ and set $c=e^{2}\alpha ,d=e^{2}\gamma $ and $x\left( \delta
		\right) =\ln \left( b+e^{p_{\min }}/\delta \right) \leq \ln \left( \left(
		b+e^{p_{\min }}\right) /\delta \right) $, so $\delta \leq \left(
		b+e^{p_{\min }}\right) e^{-x\left( \delta \right) }$. Furthermore set $%
		t\left( \delta \right) =\sqrt{cx\left( \delta \right) }+dx\left( \delta
		\right) $, so $t$ is decreasing in $\delta $. If $t\left( \delta \right)
		>t\left( 1\right) $, then $\delta \in \left( 0,1\right) $ and by (i) and
		Lemma \ref{Lemma little}%
		\begin{eqnarray*}
			\mathbb{P}\left\{ \left\vert Y\right\vert >t\left( \delta \right) \right\}
			&\leq &\delta \leq \left( b+e^{p_{\min }}\right) e^{-x\left( \delta \right) }
			\\
			&\leq &\left( b+e^{p_{\min }}\right) \exp \left( \frac{-t\left( \delta
				\right) ^{2}}{e^{2}\left( \alpha +2\gamma t\left( \delta \right) \right) }%
			\right) .
		\end{eqnarray*}%
		Since the right hand side is trivial for smaller values of $t\left( \delta
		\right) $, the inequality holds for all $t>0$. This gives (ii). (iii)\
		follows from retracing the arguments with $b=1$.\bigskip
	\end{proof}
	
	\begin{proof}[Proof of Proposition \protect\ref{Mytheorem}]
		The definitions of $V^{+}f$ and $Qf$ and the self-boundedness imply%
		\begin{eqnarray*}
			V^{+}f\left( \mathbf{x}\right) &\leq &\sum_{k=1}^{n}\left( f\left( \mathbf{x}%
			\right) -\inf_{y\in \mathcal{X}}f\left( S_{y}^{k}\mathbf{x}\right) \right)
			^{2} \\
			&\leq &\max_{k}\left( f\left( \mathbf{x}\right) -\inf_{y\in \mathcal{X}%
			}f\left( S_{y}^{k}\mathbf{x}\right) \right) \sum_{k=1}^{n}f\left( \mathbf{x}%
			\right) -\inf_{y\in \mathcal{X}}f\left( S_{y}^{k}\mathbf{x}\right) \\
			&\leq &\left( Qf\right) \left( \mathbf{x}\right) af\left( \mathbf{x}\right)
			\leq a\left( Qf\right) \left( \mathbf{x}\right) ,
		\end{eqnarray*}%
		where we used $f\left( \mathbf{x}\right) \in \left[ 0,1\right] $. The
		Efron-Stein inequality (Theorem 3.1 in \cite{Boucheron13}) then proves the
		bound on the variance. Furthermore $\left\Vert Qf\left( \mathbf{X}\right)
		\right\Vert _{q}\leq 2\lambda ^{-1}b^{1/q}q+w$ by Lemma \ref{Lemma basic}.
		Substitution in the moment inequalities of Theorem \ref{Theorem Moment
			inequalities} gives, using $\kappa \leq C$, for $q\geq 2$ the inequalities%
		\begin{equation*}
		\left\Vert \left( f\left( \mathbf{X}\right) -E\left[ f\left( \mathbf{X}%
		\right) \right] \right) _{+}\right\Vert _{q}\leq \sqrt{\kappa a\left(
			\lambda ^{-1}b^{2/q}q^{2}+wq\right) }\leq \sqrt{Ca\lambda ^{-1}}b^{1/q}q+%
		\sqrt{Cawq}
		\end{equation*}%
		and, using $a,b\geq 1$,%
		\begin{eqnarray*}
			\left\Vert \left( f\left( \mathbf{X}\right) -E\left[ f\left( \mathbf{X}%
			\right) \right] \right) _{-}\right\Vert _{q} &\leq &\sqrt{C}\left( \sqrt{%
				a\lambda ^{-1}b^{2/q}q^{2}+awq}\vee \left( 2\lambda
			^{-1}b^{1/q}q^{2}+wq\right) \right) \\
			&\leq &\sqrt{C}\left( \sqrt{a\left( \lambda ^{-1}b^{2/q}q^{2}+wq\right) }%
			\vee 2a\left( \lambda ^{-1}b^{2/q}q^{2}+wq\right) \right) \\
			&\leq &\sqrt{Ca\left( \lambda ^{-1}b^{2/q}q^{2}+wq\right) } \\
			&\leq &\sqrt{Ca\lambda ^{-1}}b^{1/p}q+\sqrt{Cawq}.
		\end{eqnarray*}%
		To see the third inequality recall that the range of $f$ is in $\left[ 0,1%
		\right] $, so the left hand side above can be at most $1$. But for any $%
		x\geq 0$ we have $\sqrt{C}\left( \sqrt{x}\vee 2x\right) \leq 1$ $\implies 
		\sqrt{x}\vee 2x\leq 1/2$ $\implies \sqrt{x}\leq 1/2$ $\implies 2x\leq \sqrt{x%
		}$ $\implies \sqrt{C}\left( \sqrt{x}\vee 2x\right) =\sqrt{Cx}$. We then use
		Lemma \ref{Lemma convert moment to confidence} with $\gamma =\sqrt{Ca\lambda
			^{-1}}$, $\alpha =Caw$, $b=b$ and $p_{\min }=2$ and a union bound to get the
		conclusion.
	\end{proof}
	
	\subsection{Miscellaneous\label{Section Proof Miscellaneous}}
	
	\begin{proposition}
		\label{Proposition Convergence}$\hat{M}\left( \mathbf{X}_{1}^{n},r\right) $
		converges to zero almost surely as $n\rightarrow \infty $.
	\end{proposition}
	
	At this point it is worth mentioning that for totally bounded $\left( 
	\mathcal{X},d\right) $ Berend and Kontorovich \cite{berend2012missing} show
	that $M\left( \mu ,n,r\right) \leq \left\vert \mathcal{C}\left( r\right)
	\right\vert /\left( en\right) $, where $\mathcal{C}\left( r\right) $ is an $%
	r $-cover in $\mathcal{X}$.
	
	\begin{lemma}
		\label{Lemma Covering}For every $r,\epsilon >0$ we can write $\mathcal{X}$
		as the disjoint union of two sets $F$ and $R$ such that $\mu \left( R\right)
		<\epsilon $ and $F$ is a finite union $F=\bigcup_{i=1}^{N}C_{i}$ where the $%
		C_{i}$ have diameter at most $r$ and $\mu \left( C_{i}\right) >0$.
	\end{lemma}
	
	\begin{proof}
		Since $\mathcal{X}$ is separable we can cover $\mathcal{X}$ with open balls $%
		\left\{ D_{i}\right\} _{i\geq 1}$ of radius $r/2$ and write $%
		C_{i}=D_{i}\backslash \bigcup_{1\leq j<i}D_{j}$. The $C_{i}$ are disjoint
		and $1=\mu \left( \mathcal{X}\right) =\sum_{i\geq 1}\mu \left( C_{i}\right) $%
		, so there is $N$ such that $\epsilon >\sum_{i\geq N+1}\mu \left(
		C_{i}\right) =\mu \left( \bigcup_{i>N}C_{i}\right) $. Set $%
		R:=\bigcup_{i>N}C_{i}\cup $ $\bigcup_{i:\mu \left( C_{i}\right) =0}C_{i}$ $%
		F=\bigcup_{1\leq i\leq N,\mu \left( C_{i}\right) >0}C_{i}$.\bigskip
	\end{proof}
	
	In the proof below we use the following consequence of the Borel-Cantelli
	lemma (\cite{bauer2011probability}): let $Y_{n}$ be a sequence of random
	variables. If for every $\epsilon >0$ we have $\sum_{n>1}\mathbb{P}\left\{
	\left\vert Y_{n}\right\vert >\epsilon \right\} <\infty $ then $%
	Y_{n}\rightarrow 0$ almost surely as $n\rightarrow \infty $.\bigskip
	
	\begin{proof}[Proof of Proposition \protect\ref{Proposition Convergence}]
		Fix $\epsilon >0$ and let $F$, $R$ and $C_{i}$ be as in Lemma \ref{Lemma
			Covering}. For each $n\in \mathbb{N}$ consider the event $A_{n}=\left\{ \hat{%
			M}\left( \mathbf{X}_{1}^{n},r\right) >\epsilon \right\} $. In case of $A_{n}$
		we have%
		\begin{eqnarray*}
			\epsilon  &<&\hat{M}\left( \mathbf{X}_{1}^{n},r\right) =\mu \left(
			\bigcap_{i=1}^{n}B\left( X_{i},r\right) ^{c}\right)  \\
			&=&\mu \left( R\cap \bigcap_{i=1}^{n}B\left( X_{i},r\right) ^{c}\right) +\mu
			\left( F\cap \bigcap_{i=1}^{n}B\left( X_{i},r\right) ^{c}\right)  \\
			&<&\epsilon +\sum_{j=1}^{N}\mu \left( \bigcap_{i=1}^{n}B\left(
			X_{i},r\right) ^{c}\cap C_{j}\right) .
		\end{eqnarray*}%
		Thus $A_{n}$ implies that there exists $C_{j}$ such that 
		\begin{equation}
		\bigcap_{i=1}^{n}\left( B\left( X_{i},r\right) ^{c}\cap C_{j}\right) \neq
		\emptyset .  \label{Borel Cantelli 1}
		\end{equation}%
		Now if there is any $X_{i}\in C_{j}$ then $C_{j}\subseteq B\left(
		X_{i},r\right) $ (by the constraint on the diameter of $C_{j}$) whence $%
		B\left( X_{i},r\right) ^{c}\cap C_{j}=\emptyset $. Thus $($\ref{Borel
			Cantelli 1}$)$ implies that for all $i\in \left[ n\right] $ we have $%
		X_{i}\notin C_{j}$. It follows that 
		\begin{eqnarray*}
			\mathbb{P}A_{n} &\leq &\mathbb{P}\bigcup_{j=1}^{N}\left\{ \mathbf{X}:\mu
			\left( \bigcap_{i=1}^{n}B\left( X_{i},r\right) ^{c}\cap C_{j}\right)
			>0\right\}  \\
			&\leq &\mathbb{P}\bigcup_{j=1}^{N}\bigcap_{i=1}^{n}\left\{ \mathbf{X}%
			:X_{i}\notin C_{j}\right\} \leq N\left( 1-\min_{j}\mu \left( C_{j}\right)
			\right) ^{n}.
		\end{eqnarray*}%
		Thus $\sum_{n}\mathbb{P}A_{n}<\infty $, and thus $\hat{M}\left( \mathbf{X}%
		_{1}^{n}\right) \rightarrow 0$ a.s.\bigskip 
	\end{proof}
	
	\begin{proposition}
		\label{Proposition example}For $p\in \left( 1,\infty \right) $ there exists
		a distribution $\mu $ in $L_{p}\left[ 0,\infty \right) $ whose support is
		not totally bounded, nowhere smooth and not contained in any finite
		dimensional subspace, but $h\left( \mathbf{X},r\right) \leq 2^{p}+1$ for any 
		$r>0$ and $\mathbf{X}\sim \mu $.
	\end{proposition}
	
	\begin{proof}
		Let $\mu $ be the distribution of the random variable $1_{\left[ 0,X\right]
		} $ in $L_{p}\left[ 0,\infty \right) $ with $X$ any real random variable
		whose distribution has full support on $\left[ 0,\infty \right) $ (the
		exponential distribution would do). It is easy to see that the support of $%
		\mu $ has the required properties. Then note that $\left\Vert 1_{\left[ 0,a%
			\right] }-1_{\left[ 0,b\right] }\right\Vert _{p}=\left\vert a-b\right\vert
		^{1/p}$, so if $h\left( \mathbf{X},r\right) \geq k$ then $\exists f\in L_{p}%
		\left[ 0,\infty \right) $ and $x_{1},...,x_{k}\in \left[ 0,1\right] $ with $%
		x_{i-1}<x_{i}$, $\left\Vert 1_{\left[ 0,x_{i}\right] }-1_{\left[ 0,x_{i-1}%
			\right] }\right\Vert _{p}>r$ and $\left\Vert 1_{\left[ 0,x_{i}\right]
		}-f\right\Vert _{p}\leq r$. Then $2r\geq \left\Vert 1_{\left[ 0,x_{1}\right]
	}-1_{\left[ 0,x_{k}\right] }\right\Vert _{p}=\left\vert
	x_{k}-x_{1}\right\vert ^{1/p}=\left( \sum_{i=2}^{k}\left(
	x_{i}-x_{i-1}\right) \right) ^{1/p}>\left( k-1\right) ^{1/p}r$, so $k-1<2^{p}
	$. \bigskip
\end{proof}

\begin{proposition}
	\label{Proposition packing number Restatement}Let $\left( \mathbb{R}%
	^{D},\left\Vert \text{.}\right\Vert \right) $ be a finite dimensional Banach
	space with closed unit ball $\mathbb{B}$ and define the $1$-packing number
	of $\mathbb{B}$ as 
	\begin{equation*}
	\mathcal{P}\left( \mathbb{B},d_{\left\Vert .\right\Vert },1\right) :=\max
	\left\{ \left\vert S\right\vert :S\subset \mathbb{B}^{D},\forall x,y\in
	S,x\neq y\implies \left\Vert x-y\right\Vert >1\right\} .
	\end{equation*}%
	Let $r>0$. Then
	
	(i) for every vector $\mathbf{x}\in \left( \mathbb{R}^{D}\right) ^{n}$ we
	have $h\left( \mathbf{x},r\right) \leq \mathcal{P}\left( \mathbb{B}%
	,d_{\left\Vert .\right\Vert },1\right) \leq 8^{D}$.
	
	(ii) For the $2$-norm the bound improves to $3^{D}$.
	
	(iii) If $\mu $ has a positive density w.r.t. Lebesgue measure on $\mathbb{R}%
	^{D}$ and $\mathbf{X}_{1}^{n}\sim \mu ^{n}$ then $h\left( \mathbf{X}%
	_{1}^{n},r\right) \rightarrow \mathcal{P}\left( \mathbb{B},d_{\left\Vert
		.\right\Vert },1\right) $ almost surely as $n\rightarrow \infty $.
\end{proposition}

\begin{proof}
	(i) Let $\mathbf{z}=\left( z_{1},...,z_{m}\right) \subseteq \mathbf{x}$
	satisfy the local separation property with $h\left( \mathbf{x},r\right) =m$.
	So there is $y\in \mathbb{R}^{D}$ such that $\left\Vert z_{i}-y\right\Vert
	\leq r$ and $\left\Vert z_{i}-z_{j}\right\Vert >r$ for all $i\neq j$. Let $%
	z_{i}^{\prime }=\left( 1/r\right) \left( z_{i}-y\right) $. Then $%
	z_{i}^{\prime }\in \mathbb{B}$ and $\left\Vert z_{i}^{\prime }-z_{j}^{\prime
	}\right\Vert >1$. This is the first inequality of (i). The second follows
	from Proposition 5 in \cite{Cucker01}.
	
	(ii) This follows from the first inequality in (i) and Proposition 4.2.12 in 
	\cite{vershynin2018high}.
	
	(ii) Let $B\left( y,r\right) $ be any ball of radius $r$ in $\mathbb{R}^{D}$%
	, $\mathbf{z}=\left( z_{1},...,z_{K}\right) $ be any $r$-separated vector of
	points in $B\left( y,r\right) $ with $K=\mathcal{P}\left( B\left( y,r\right)
	,d_{\left\Vert .\right\Vert },r\right) =\mathcal{P}\left( \mathbb{B}%
	,d_{\left\Vert .\right\Vert },1\right) $. Since the separation condition is
	defined by strict inequalities, there is some $\eta >0$ such that every
	vector $\mathbf{z}^{\prime }=\left( z_{1}^{\prime },...,z_{K}^{\prime
	}\right) $ satisfying $z_{k}^{\prime }\in B\left( z_{k},\eta \right) $ for
	all $k\in \left[ K\right] $, is also $r$-separated. Since $\mu $ has a
	positive density w.r.t. Lebesgue measure $\mu \left( B\left( z_{k},\eta
	\right) \right) >0$ for each $k$.
\end{proof}

Now let $A_{n}$ be the event $A_{n}=\left\{ \left\vert \mathcal{P}\left( 
\mathbb{B},d_{\left\Vert .\right\Vert },1\right) -h\left( \mathbf{X}%
_{1}^{n},r\right) \right\vert >\epsilon \right\} $. Since $h\left( \mathbf{X}%
_{1}^{n},r\right) \leq \mathcal{P}\left( \mathbb{B},d_{\left\Vert
	.\right\Vert },1\right) $ (by (i)), under $A_{n}$ there must exist $k\in %
\left[ K\right] $, such that for all $i\in \left[ n\right] $, $X_{i}\notin
B\left( z_{k},\eta \right) $. Thus 
\begin{equation*}
\mathbb{P}\left( A_{n}\right) \leq K\left( 1-\min_{k}\lambda \left( B\left(
z_{k},\eta \right) \right) \right) ^{n}
\end{equation*}%
and the conclusion follows from the Borel-Cantelli lemma, as in Proposition %
\ref{Proposition Convergence}.\bigskip 

\subsection{The Wasserstein distance\label{Section Proof Wasserstein}}

\begin{theorem}
	\label{Theorem Wasserstein Restatement}( \textbf{= Theorem \ref{Theorem
			Wasserstein}) }Let $\left( \mathcal{X},d\right) $ be a complete, separable
	metric space with diameter $1$ and Borel probability measure $\mu $. For $%
	\delta >0$, with probability at least $1-\delta $ in $\mathbf{X}\sim \mu ^{n}
	$, if there exists an $r$-net $\mathbf{Y}\subset \mathbf{X}$ with
	cardinality $m$, $m\leq \left( n-3\right) /2$, then%
	\begin{equation*}
	W_{1}\left( \mu ,\hat{\mu}\right) \leq \hat{M}\left( \mathbf{X},r\right)
	+3r+2\sqrt{\frac{m}{n-m}}\left( 1+\sqrt{\ln \left( n/\delta \right) }\right) 
	\end{equation*}%
	or%
	\begin{equation*}
	W_{1}\left( \mu ,\hat{\mu}\right) \leq 3r+3\sqrt{\frac{m}{n-m}}\left( 1+%
	\sqrt{\ln \left( 2n/\delta \right) }\right) .
	\end{equation*}
\end{theorem}

\begin{proof}
	Let $V:\mathbf{y}=\left( y_{1},...,y_{m}\right) \in \mathcal{X}%
	^{m}\rightarrow \left( V_{1},...,V_{m}\right) \in \Sigma ^{m}$ be the
	Voronoi partitioning associated with $\mathbf{y}$ and tie breaking according
	to the order of indices in $\mathbf{y}$. Define $E\left( \mathbf{y}\right)
	_{k}=V\left( \mathbf{y}\right) _{k}\cap B\left( y_{k},2r\right) $. Note that
	the $E\left( \mathbf{y}\right) _{k}$ are disjoint and%
	\begin{equation}
	\bigcup_{k=1}^{m}B\left( y_{k},2r\right) =\bigcup_{k=1}^{m}E\left( \mathbf{y}%
	\right) _{k}\text{.}  \label{Union equality}
	\end{equation}%
	Let $I\subseteq \left[ n\right] $ be a set of indices with $\left\vert
	I\right\vert =m$ and assume that $\mathbf{Y}=\left( X_{i}\right) _{i\in I}$
	is an $r$-net. We write $\mathbf{Y}=\left( Y_{1},...,Y_{m}\right) $ and%
	\begin{equation*}
	\hat{\mu}_{\mathbf{X\backslash Y}}=\frac{1}{n-m}\sum_{i\in \left[ n\right]
		\backslash I}\delta _{X_{i}}
	\end{equation*}%
	and define an intermediate probability measure 
	\begin{equation*}
	\bar{\mu}=\sum_{k=1}^{m}\hat{\mu}_{\mathbf{X\backslash Y}}\left( E\left( 
	\mathbf{Y}\right) _{k}\right) \delta _{Y_{k}}.
	\end{equation*}%
	We will use the triangle inequality to bound $W_{1}\left( \mu ,\hat{\mu}%
	\right) \leq W_{1}\left( \bar{\mu},\hat{\mu}\right) +W_{1}\left( \mu ,\bar{%
		\mu}\right) $. We begin with the first term.
	
	Note that at most one $Y_{k}$ can be in $E\left( \mathbf{Y}\right) _{k}$, so%
	\begin{equation*}
	\hat{\mu}_{\mathbf{X\backslash Y}}\left( E\left( \mathbf{Y}\right)
	_{k}\right) \leq \frac{\left\vert \left\{ i:X_{i}\in E\left( \mathbf{Y}%
		\right) _{k}\right\} \right\vert -1}{n-m}.
	\end{equation*}%
	Let $\mathcal{F}$ be the class of all $f:\mathcal{X\rightarrow }\left[ 0,1%
	\right] $ satisfying $\left\Vert f\right\Vert _{Lip}\leq 1$. Since diam$%
	\left( \mathcal{X}\right) \leq 1$, by replacing the Lipschitz function $f$
	in in (\ref{Kantorovich Rubinstein}) by $f-\inf_{x\in \mathcal{X}}f\left(
	x\right) $ we can always assume that $f\in \mathcal{F}$. Then for any such $f
	$%
	\begin{eqnarray}
	\int_{\mathcal{X}}f\left( d\bar{\mu}-d\hat{\mu}\right) 
	&=&\sum_{k=1}^{m}\left( f\left( Y_{k}\right) \hat{\mu}_{\mathbf{X\backslash Y%
		}}\left( E\left( \mathbf{Y}\right) _{k}\right) -\frac{1}{n}\sum_{i:X_{i}\in
		E\left( \mathbf{Y}\right) _{k}}f\left( Y_{k}\right) \right)   \notag \\
	&&\text{ \ \ \ \ \ \ \ \ \ \ \ \ \ \ \ \ \ }+\frac{1}{n}\sum_{k=1}^{m}%
	\sum_{i:X_{i}\in E\left( \mathbf{Y}\right) _{k}}\left( f\left( Y_{k}\right)
	-f\left( X_{k}\right) \right) .  \notag
	\end{eqnarray}%
	The second term is bounded by $r$ by the Lipschitz condition on $f$. The
	first term is%
	\begin{eqnarray*}
		&&\sum_{k=1}^{m}\left( f\left( Y_{k}\right) \left( \frac{\left\vert \left\{
			i:X_{i}\in E\left( \mathbf{Y}\right) _{k}\right\} \right\vert -1}{n-m}-\frac{%
			\left\vert \left\{ i:X_{i}\in E\left( \mathbf{Y}\right) _{k}\right\}
			\right\vert }{n}\right) \right)  \\
		&\leq &\sum_{k=1}^{m}\left\vert \frac{\left\vert \left\{ i:X_{i}\in E\left( 
			\mathbf{Y}\right) _{k}\right\} \right\vert -1}{n-m}-\frac{\left\vert \left\{
			i:X_{i}\in E\left( \mathbf{Y}\right) _{k}\right\} \right\vert }{n}%
		\right\vert  \\
		&=&\sum_{k=1}^{m}\left\vert \frac{m\left\vert \left\{ i:X_{i}\in E\left( 
			\mathbf{Y}\right) _{k}\right\} \right\vert -n}{\left( n-m\right) n}%
		\right\vert  \\
		&\leq &\sum_{k=1}^{m}\frac{m\left\vert \left\{ i:X_{i}\in E\left( \mathbf{Y}%
			\right) _{k}\right\} \right\vert +n}{\left( n-m\right) n}=\frac{m+1}{n-m}%
		\text{.}
	\end{eqnarray*}%
	Thus 
	\begin{equation}
	W_{1}\left( \bar{\mu},\hat{\mu}\right) \leq \frac{m+1}{n-m}+r\leq \sqrt{%
		m/\left( n-m\right) }+r  \label{Bound mu_bar mu_hat}
	\end{equation}
	by virtue of the condition $m\leq \left( n-3\right) /2.$ Now define%
	\begin{equation*}
	\Phi \left( \mathbf{X\backslash Y}\right) :=\sum_{k=1}^{m}\left\vert \mu
	\left( E\left( \mathbf{Y}\right) _{k}\right) -\hat{\mu}_{\mathbf{X\backslash
			Y}}\left( E\left( \mathbf{Y}\right) _{k}\right) \right\vert .
	\end{equation*}%
	From (\ref{Union equality}) and the fact, that the $E\left( \mathbf{Y}%
	\right) _{k}$ are mutually disjoint, we also obtain for $f\in \mathcal{F}$%
	\begin{eqnarray}
	\int_{\mathcal{X}}f\left( d\bar{\mu}-d\hat{\mu}\right)  &\leq
	&\int_{\bigcap_{k}B\left( Y_{k},2r\right) ^{c}}fd\mu
	+\sum_{k=1}^{m}\int_{E\left( \mathbf{Y}\right) _{k}}f\left( d\mu -d\bar{\mu}%
	\right)   \notag \\
	&\leq &\hat{M}\left( \mathbf{X},r\right) +\sum_{k=1}^{m}\int_{E\left( 
		\mathbf{Y}\right) _{k}}\left( f-f\left( Y_{k}\right) \right) d\mu
	+\sum_{k=1}^{m}f\left( Y_{k}\right) \left( \mu \left( E\left( \mathbf{Y}%
	\right) _{k}\right) -\bar{\mu}\left( E\left( \mathbf{Y}\right) _{k}\right)
	\right)   \notag \\
	&\leq &\hat{M}\left( \mathbf{X},r\right) +2r+\Phi \left( \mathbf{X\backslash
		Y}\right) ,  \label{Bound mu mu_bar}
	\end{eqnarray}%
	where the second inequality follows from the triangle inequality and the
	fact that $\mathbf{Y}$ is an $r$-net of $\mathbf{X}$, so that $%
	\bigcap_{k=1}^{m}B\left( Y_{k},2r\right) ^{c}\subseteq
	\bigcap_{i=1}^{n}B\left( X_{i},r\right) ^{c}$. The last inequality comes
	from the Lipschitz property of $f$, since $d\left( x,Y_{k}\right) \leq 2r$
	for all $x\in E\left( \mathbf{Y}\right) _{k}$. 
	
	From Jensen's inequality we get%
	\begin{eqnarray*}
		\mathbb{E}\left[ \Phi \left( \mathbf{X\backslash Y}\right) |\mathbf{Y}\right]
		&\leq &\sum_{k=1}^{m}\left( \mathbb{E}\left( \frac{1}{n-m}\sum_{j\notin I_{%
				\mathbf{Y}}}\left( \mu \left( E\left( \mathbf{Y}\right) _{k}\right)
		-1\left\{ X_{j}\in E\left( \mathbf{Y}\right) _{k}\right\} \right) \right)
		^{2}\right) ^{1/2} \\
		&=&\frac{1}{n-m}\sum_{k=1}^{m}\left( \sum_{j\notin I_{\mathbf{Y}}}\mathbb{E}%
		\left( \mu \left( E\left( \mathbf{Y}\right) _{k}\right) -1\left\{ X_{j}\in
		E\left( \mathbf{Y}\right) _{k}\right\} \right) ^{2}\right) ^{1/2} \\
		&\leq &\frac{1}{\sqrt{n-m}}\sum_{k=1}^{m}\sqrt{\mu \left( E\left( \mathbf{Y}%
			\right) _{k}\right) }\leq \sqrt{\frac{m}{n-m}}.
	\end{eqnarray*}
	
	On the other hand modifying $X_{i}$ for some $i\notin \mathbf{Y}$ can change
	the value of $\Phi \left( \mathbf{X\backslash Y}\right) $ at most for two
	values of $k$, since the $E\left( \mathbf{Y}\right) _{k}$ are mutually
	disjoint, so the incured difference in $\Phi \left( \mathbf{X\backslash Y}%
	\right) $ is bounded by $2/\left( n-m\right) $. It follows from the bounded
	difference inequality that $\mathbb{P}\left\{ \Phi \left( \mathbf{%
		X\backslash Y}\right) -\mathbb{E}\left[ \Phi \left( \mathbf{X\backslash Y}%
	\right) |\mathbf{Y}\right] >t\right\} \leq \exp \left( -\left( n-m\right)
	t^{2}/4\right) $. We conclude from (\ref{Bound mu_bar mu_hat}) and (\ref%
	{Bound mu mu_bar}) that for any fixed $r$-net $\mathbf{Y\subset X}$ with
	cardinality $m\leq \left( n-3\right) /2$ and $t>0$ 
	\begin{align*}
	& \mathbb{P}\left\{ W_{1}\left( \mu ,\hat{\mu}\right) >\hat{M}\left( \mathbf{%
		X},r\right) +3r+2\sqrt{\frac{m}{n-m}}+t\right\}  \\
	& \leq \mathbb{P}\left\{ \Phi \left( \mathbf{X\backslash Y}\right) -\mathbb{E%
	}\left[ \Phi \left( \mathbf{X\backslash Y}\right) |\mathbf{Y}\right]
	>t\right\}  \\
	& \leq \exp \left( -\left( n-m\right) t^{2}/4\right) .
	\end{align*}%
	A union bound over all the $\binom{n}{m}\leq n^{m}$ sub-samples $\mathbf{Y}%
	\subset \mathbf{X}$ with $\left\vert \mathbf{Y}\right\vert =m$ gives 
	\begin{equation*}
	\mathbb{P}\left\{ \exists \mathbf{Y}\subseteq \mathbf{X},\left\vert \mathbf{Y%
	}\right\vert =m,W_{1}\left( \mu ,\hat{\mu}\right) >\hat{M}\left( \mathbf{X}%
	,r\right) +3r+2\sqrt{\frac{m}{n-m}}+t\right\} \leq n^{m}\exp \left( -\left(
	n-m\right) t^{2}/4\right) .
	\end{equation*}%
	The first conclusion follows from equating the bound on the probability to $%
	\delta $ and solving for $t$.The second conclusion follows from a union
	bound with Corollary \ref{Corollary bound missing mass with net}.
\end{proof}

\bibliographystyle{plain}

\end{document}